\documentclass[12pt]{amsart}
\usepackage[margin=1.2in]{geometry}
\usepackage{graphicx,latexsym}
\usepackage{tikz-cd}
\usepackage{tikz}
\usepackage{amsfonts, amssymb, amsmath, amsthm, bm}

\newtheorem{theorem}{Theorem}[section]
\newtheorem{proposition}[theorem]{Proposition}
\newtheorem{lemma}[theorem]{Lemma}

\theoremstyle{definition}

\theoremstyle{remark}
\newtheorem{remark}[theorem]{Remark}

 \numberwithin{equation}{section}

\newcommand{\N}{\mathbb N}

\newcommand{\C}{\mathbb C}
\newcommand{\R}{\mathbb R}
\newcommand{\Z}{\mathbb Z}

\usetikzlibrary{arrows,chains,matrix,positioning,scopes}
\makeatletter
\tikzset{join/.code=\tikzset{after node path={%
\ifx\tikzchainprevious\pgfutil@empty\else(\tikzchainprevious)%
edge[every join]#1(\tikzchaincurrent)\fi}}}
\makeatother
\tikzset{>=stealth',every on chain/.append style={join},
         every join/.style={->}}
\tikzstyle{labeled}=[execute at begin node=$\scriptstyle,
   execute at end node=$]

\begin{document}
\title[The Cousin problem for quasianalytic functions]{Solution to the first Cousin problem for vector-valued quasianalytic functions}
\author[A. Debrouwere]{Andreas Debrouwere}
\address{Department of Mathematics, Ghent University, Krijgslaan 281, 9000 Gent, Belgium}
\email{Andreas.Debrouwere@UGent.be}
\thanks{A. Debrouwere gratefully acknowledges support by Ghent University, through a BOF Ph.D.-grant.}

\author[J. Vindas]{Jasson Vindas}
\thanks{The work of J. Vindas was supported by Ghent University, through the BOF-grant 01N01014.}
\address{Department of Mathematics, Ghent University, Krijgslaan 281, 9000 Gent, Belgium}
\email{Jasson.Vindas@UGent.be}

\subjclass[2010]{Primary 30D60, 46E10, 46E40. Secondary 46M05, 46M18.}
\keywords{Cousin problem; Vector-valued quasianalytic functions; Homological methods in functional analysis}
\begin{abstract}
We study spaces of vector-valued quasianalytic functions and solve the first Cousin problem in these spaces.
\end{abstract}

\maketitle
\section{Introduction}
In abstract terms, the first Cousin problem can be formulated as follows. Let $X$ be a topological space and let $\mathcal{F}$ be a sheaf on $X$. Let $\Omega \subseteq X$ be open and let $\mathcal{M} = \{ \Omega_i \, : \, i \in I\}$ be an open covering of $\Omega$. Suppose that $\varphi_{i,j} \in \Gamma(\Omega_i \cap \Omega_j, \mathcal{F})$, $i,j \in I$, are given sections such that
$$
\varphi_{i,j} + \varphi_{j,k} + \varphi_{k,i} = 0 \qquad \mbox{on }  \Omega_i \cap \Omega_j \cap \Omega_k, 
$$
for all $i,j,k \in I$. Are there $\varphi_{i} \in \Gamma(\Omega_i, \mathcal{F})$, $i \in I$, such that
$$
\varphi_{i,j} = \varphi_{j} - \varphi_{i} \qquad \mbox{on }  \Omega_i \cap \Omega_j,
$$
for all $i,j \in I$? For $X = \C^d$ and $\mathcal{F}$ the sheaf of holomorphic functions the Cousin problem is solvable if $\Omega$ is a Stein open set, as follows from the celebrated Oka-Cartan theorem. This problem was very important for the development of the modern theory of functions of several complex variables and led to the use of sheaf cohomology in that area. We refer to \cite{Gunning} for a clear exposition of the problem. Since every open set in $\R^d$ has a system of complex neighborhoods consisting of Stein open sets, it follows that the Cousin problem is solvable for $X = \R^d$ and $\mathcal{F}$ the sheaf of real analytic functions (where $\Omega$ is now an arbitrary open set). Petzsche announced in \cite{Petzsche1984} the solution to the Cousin problem for quasianalytic classes in connection with the construction of sheaves of infrahyperfunctions, but his article on the subject seems not to have appeared.

The aim of this paper is to show that the Cousin problem is in fact solvable for spaces of quasianalytic functions. We shall also give sufficient conditions on a locally convex space $F$ such that the Cousin problem is solvable in spaces of $F$-valued quasianalytic functions. 
 We mention that in a forthcoming paper \cite{DVVinfra} the authors will apply the vector-valued results from this article to construct sheaves of differential algebras in which the spaces of infrahyperfunctions of class $\{M_p\}$ \cite{Hormander} are embedded in such a way that the ordinary multiplication of ultradifferentiable functions of class $\{M_p\}$ is preserved. Notice that for $M_p = p!$, one obtains a differential algebra that contains the space of all hyperfunctions and in which the multiplication of real analytic functions coincides with their pointwise product. The construction of such algebras and embeddings has been an important and long-standing open question in the non-linear theory of generalized functions.

The analysis of the Cousin problem requires the study of topological properties of the spaces of quasianalytic functions. The space of real analytic functions has been thoroughly investigated in the literature and its locally convex structure is by now well understood; see \cite{Bonet, Domanski4, Domanski3, Martineau2} for the scalar-valued case and \cite{Bonet00, BDV, Domanski2, K-M} for the vector-valued case. This is much less the case for other spaces of quasianalytic functions, although some work has been done \cite{Bonet, Chung, K-M-R, Rosner}. The first part of this article is devoted to studying various useful topological properties of spaces of vector-valued quasianalytic  functions (defined via weight sequences \cite{Komatsu}). Even in the scalar-valued case, some of the results we discuss here appear to be new; for example, we will establish that the spaces of quasianalytic functions of Roumieu type are ultrabornological $(PLN)$-spaces, a fact that is crucial for us to solve the Cousin problem in this case and that, to the best of our knowledge, remained unnoticed in the literature for general open subsets of $\mathbb{R}^{d}$ (see Remark \ref{rkCousin1}).  

The plan of the paper is as follows. Section \ref{section: projective/inductive spectra} explains some basic material on locally convex spaces that we shall need later. In Section \ref{section: scalar valued quasianalytic functions} we prove that the spaces of quasianalytic functions of Roumieu type are ultrabornological with the aid of H\"{or}mander's support theorem for quasianalytic functionals \cite{Heinrich,Hormander}. A generalization of Komatsu's first structure theorem for quasianalytic functionals is discussed in Section \ref{sect-structure}. This result enables us to give an explicit system of seminorms generating the topology on the quasianalytic function spaces of Roumieu type; such a projective topological description plays an important role in the analysis of the vector-valued case. We study vector-valued quasianalytic functions in Section \ref{sect-vector}, we closely follow there Komatsu's approach from \cite{Komatsu3}; in order to discuss their topological properties, we make use of the dual interpolation estimate for the space of real analytic functions \cite{Bonet} and a deep result of Doma\'nski on the $\varepsilon$-product of $(PLS)$-spaces \cite{Domanski2}. The Cousin problem is solved in Section \ref{sect-Cousin}. Our proof is based on duality theory and the vanishing of the $\operatorname{Proj}^1$-functor for ultrabornological $(PLS)$-spaces. The result is extended to the vector-valued case by using the topological properties obtained in Section \ref{sect-vector}. 

We are indebted to the authors of \cite{Bonet,Domanski2},  as many of our proofs below rely on their results. In particular, Doma\'nski's work on the $\varepsilon$-product of $(PLS)$-spaces was very inspiring to us.

\section{Projective and inductive spectra of locally convex spaces}
\label{section: projective/inductive spectra}

In this preliminary section we collect some useful background material  on projective and inductive spectra of locally convex spaces that will be used in the next sections. Of particular importance for us is the characterization of ultrabornological $(PLS)$-spaces due to Vogt and Wengenroth \cite{Wengenroth} that we state below.

Throughout this article every locally convex space (from now on abbreviated as l.c.s.) is assumed to be Hausdorff. Given a l.c.s.\ $X$ we write $X'$ for its topological dual. Unless otherwise stated, we endow $X'$ with the strong topology.
A \emph{projective spectrum} is a sequence $\mathcal{X} = (X_n, \iota^n_{n+1})_{n \in \N}$ consisting of vector spaces $X_n$  and linear 
mappings $\iota^n_{n+1}: X_{n+1} \rightarrow X_n$. 
Set
$$ 
\operatorname{Proj}^0 \mathcal{X} =  \varprojlim_{n \in \N} X_n
$$
and denote  by $\iota_k$, $k \in \N$, the canonical mapping of $\operatorname{Proj}^0 \mathcal{X}$ into $X_k$. 
Define
$$ \operatorname{Proj}^1 \mathcal{X} =   \prod_{n \in \N}X_n / B(\mathcal{X}),$$
where
$$B(\mathcal{X}) = \{(x_n) \in \prod_{n \in \N} X_n \, : \exists (y_n) \in \prod_{n \in \N} X_n \mbox { with }  x_n = y_n - \iota^n_{n+1}(y_{n+1}), \quad \forall n \in \N\}.$$
This definition is due to Palamodov \cite{Palamodov} and coincides with his original definition in terms of homological algebra (see \cite[Sect.\ 3.1]{Wengenroth2}). Let 
\begin{center}
\begin{tikzpicture}
  \matrix (m) [matrix of math nodes, row sep=2em, column sep=2em]
    {0 & \mathcal{X} & \mathcal{Y} & \mathcal{Z} & 0 \\
   };
  { [start chain] \chainin (m-1-1);
\chainin (m-1-2);
\chainin (m-1-3) [join={node[above,labeled] {}}];
\chainin (m-1-4)[join={node[above,labeled] {}}];
\chainin (m-1-5)[join={node[above,labeled] {}}];
}

\end{tikzpicture}
\end{center}
be an exact sequence (in the category of projective spectra) and suppose that $\operatorname{Proj}^1 \mathcal{X} = 0$, then
\begin{center}
\begin{tikzpicture}
  \matrix (m) [matrix of math nodes, row sep=2em, column sep=2em]
    {0 & \operatorname{Proj}^0\mathcal{X} & \operatorname{Proj}^0\mathcal{Y} & \operatorname{Proj}^0\mathcal{Z} & 0 \\
   };
  { [start chain] \chainin (m-1-1);
\chainin (m-1-2);
\chainin (m-1-3) [join={node[above,labeled] {}}];
\chainin (m-1-4)[join={node[above,labeled] {}}];
\chainin (m-1-5)[join={node[above,labeled] {}}];
}

\end{tikzpicture}
\end{center}is again exact. 

A \emph{projective spectrum of l.c.s.}\ is a projective spectrum $\mathcal{X} = (X_n, \iota^n_{n+1})_n$ consisting of l.c.s.\ $X_n$ and continuous linking mappings $\iota^n_{n+1}$. 
The spectrum $\mathcal{X}$ is called \emph{reduced} if the mappings $\iota_k$ have dense range for each $k \in \N$.  

An  \emph{inductive spectrum of l.c.s.}\ is a sequence $\mathcal{X} = (X_n, \sigma^n_{n+1})_{n \in \N}$ of l.c.s.\ $X_n$ and  linear continuous mappings $\sigma^n_{n+1}: X_n \rightarrow X_{n
+1}$. The spectrum $\mathcal{X}$ is called \emph{injective} if $\sigma^n_{n+1}$ is injective for each $n \in \N$.
Set
$$ 
X = \varinjlim_{n \in \N} X_n. 
$$
Denote by $\sigma_k$, $k \in \N$, the canonical mapping of $X_k$ into $X$. The inductive spectrum $\mathcal{X}$ is called \emph{regular} ($\alpha$\emph{-regular}, resp.) if $X$ is Hausdorff and for every bounded set $B$ in $X$ there is $k \in \N$ and a bounded set $A$ in $X_k$ such that $\sigma_k(A) = B$ (there is a - not necessarily bounded - set $A$ in $X_k$ such that $\sigma_k(A) = B$). 

Let  $\mathcal{X} = (X_n, \iota^n_{n+1})_{n \in \N}$ be a projective spectrum of l.c.s.\ and set $X = \varprojlim X_n$. One defines its dual inductive spectrum as $\mathcal{X^*} = (X'_n, {}^t\iota^n_{n+1})_{n 
\in \N}$. If $\mathcal{X}$ is reduced, we have  $X' \cong \varinjlim X'_n$ as vector spaces \cite[p.\ 290]{Kothe}. If, additionally, the spaces $X_n$ are semi-reflexive, the above isomorphism also holds topologically \cite[pp.\ 294 and 300]{Kothe}. Similarly, let  $\mathcal{X} = (X_n, \sigma^n_{n+1})_{n \in \N}$ be an 
inductive spectrum of l.c.s.\ and set $X = \varinjlim X_n$. We define its dual projective spectrum as $\mathcal{X^*} = (X'_n, {}^t\sigma^n_{n+1})_{n \in \N}$.  We always have 
$X' \cong \varprojlim X'_n$ as vector spaces \cite[p.\ 290]{Kothe}. If, moreover, the spectrum $\mathcal{X}$ is regular the above isomorphism also holds topologically. 

A l.c.s.\  $X$ is called a $(PLS)$-space ($(PLN)$-space, resp.) if $X = \varprojlim X_n$
with $(X_n)_{n \in \N}$  a  projective spectrum of $(DFS)$-spaces ($(DFN)$-spaces).
A l.c.s.\  $X$ is said to be an $(LFS)$-space ($(LFN)$-space, resp.) if  $X = \varinjlim X_n$
with $(X_n)_{n \in \N}$  an injective inductive spectrum consisting of $(FS)$-spaces ($(FN)$-spaces).  
The hereditary properties of nuclearity imply that $(PLN)$-spaces are nuclear \cite[Prop.\ 50.1]{Treves}. Vogt and Wengenroth characterized ultrabornological (PLS)-spaces in 
the following way:
\begin{proposition}\cite[Thm.\ 3.3, Thm.\ 3.5]{Wengenroth} \label{VogtandWeng}
Let $\mathcal{X} = (X_n, \iota^n_{n+1})_{n \in \N}$ be a reduced projective spectrum of $(DFS)$-spaces and set 
$X = \varprojlim X_n$. Then, the following statements are equivalent:
\begin{enumerate}
\item[$(i)$] $\operatorname{Proj}^1\mathcal{X} = 0$.
\item[$(ii)$] $X$ is ultrabornological.
\item[$(iii)$] $\mathcal{X^*}$ is $\alpha$-regular. 
\item[$(iv)$] $\mathcal{X^*}$ is regular. 
\end{enumerate}
\end{proposition}

\section{Spaces of quasianalytic functions and their duals}\label{section: scalar valued quasianalytic functions}
We now discuss some topological properties of the spaces of quasianalytic functions. We work with ultradifferentiability as defined in \cite{Komatsu}.

Let $(M_p)_{p \in \N}$ be a sequence of positive real numbers and define $m_p := M_p / M_{p -1}$, $p \in \Z_+$. We call $M_p$ a weight sequence if $M_0 = 1$ and $\lim_{p \to \infty} m_p = \infty$.  We make use of the following conditions:
\begin{itemize}
\item [$(M.1)\:$] $M^{2}_{p}\leq M_{p-1}M_{p+1},$  $p\geq 1$,
\item [$(M.2)'$]$M_{p+1}\leq A H^{p+1} M_p$, $p\in\mathbb{N}$, for some $A,H \geq 1$,

\item [$(QA)$\:] $\displaystyle \sum_{p=1}^{\infty}\frac{1}{m_{p}}=\infty. $
\end{itemize}

For $\alpha \in \N^d$ we write $M_\alpha = M_{|\alpha|}$. The associated function of $M_p$ is defined as
$$
M(t)=\sup_{p\in\mathbb{N}}\log\frac{t^p}{M_p},\qquad t > 0,
$$
and $M(0)=0$. We extend $M$ to $\C^d$ as $M(z) = M(|z|)$, $z \in \C^d$. As usual \cite{Komatsu}, the relation $M_p\subset N_p$ between two weight sequences means that there are $C,h>0$ such that 
$M_p\leq Ch^{p}N_{p},$ $p\in\mathbb{N}$. The stronger relation $M_p\prec N_p$ means that the latter inequality remains valid for every $h>0$ and a suitable $C=C_{h}>0$. 

Let $K$ be a regular compact set in $\R^d$, that is, $\overline{\operatorname{int} K} = K$. For $h > 0$ we write $\mathcal{E}^{M_p,h}(K)$ for the Banach space of all $\varphi \in C^\infty(K)$ such that
$$
\| \varphi \|_{K,h} := \sup_{\alpha \in \N^d}\sup_{x \in K} \frac{|\varphi^{(\alpha)}(x)|}{h^{|\alpha|}M_{\alpha}} < \infty. 
$$
For an open set $\Omega$ in $\R^d$ we define
$$
\mathcal{E}^{(M_p)}(\Omega) = \varprojlim_{K \Subset \Omega} \varprojlim_{h \rightarrow 0^+}\mathcal{E}^{M_p,h}(K), \qquad  \mathcal{E}^{\{M_p\}}(\Omega) = \varprojlim_{K \Subset \Omega} \varinjlim_{h \rightarrow \infty}\mathcal{E}^{M_p,h}(K).
$$
The elements of their dual spaces $\mathcal{E}'^{(M_p)}(\Omega)$ and $\mathcal{E}'^{\{M_p\}}(\Omega)$ are called \emph{quasianalytic functionals of class $(M_p)$ or Beurling type in} $\Omega$ and \emph{quasianalytic functionals of class $\{M_p\}$ or Roumieu type in} $\Omega$, respectively. Notice that $\mathcal{E}^{\{p!\}}(\Omega)$ is precisely the space $\mathcal{A}(\Omega)$ of real analytic functions in $\Omega$, while $\mathcal{A}'(\Omega)$ is that of analytic functionals in $\Omega$. 

In the sequel we shall write $\ast$ instead of $(M_p)$ or $\{M_p\}$ if we want to treat both cases simultaneously. In addition, we shall often first state assertions for the $(M_p)$-case followed in parenthesis by the corresponding statements for the $\{M_p\}$-case.

We also need the ensuing assumption on $M_p$:
\begin{itemize}
\item [$(NA)$] $p! \prec M_p$ 
\end{itemize}
in the Beurling case and
\begin{itemize}
\item [$(NE)$] $p! \subset M_p$ 
\end{itemize}
in the Roumieu case. Conditions $(NA)$ and $(NE)$ guarantee that the space of entire functions is dense in $\mathcal{E}^*(\Omega)$ \cite[Prop.\ 3.2]{Hormander} \footnote{H\"ormander actually only considers the Roumieu case but his proof can easily be adapted to cover the Beurling case as well.}. Hence for $\Omega' \subseteq \Omega$ and two weight sequences $M_p$ with $N_p \subset M_p$ we may identify $\mathcal{E}'^{(M_p)}(\Omega')$ ($\mathcal{E}'^{\{M_p\}}(\Omega')$) with a vector subspace of $\mathcal{E}'^{(N_p)}(\Omega)$ ($\mathcal{E}'^{\{N_p\}}(\Omega)$). If  $N_p \prec M_p$ we have that $\mathcal{E}'^{(M_p)}(\Omega') \subset  \mathcal{E}'^{\{N_p\}}(\Omega)$. In particular, we always have $\mathcal{E}'^\ast(\Omega') \subseteq \mathcal{A}'(\Omega)$. 

Unless otherwise explicitly stated, $M_p$ will \emph{always} stand for a weight sequence satisfying $(M.1)$, $(M.2)'$, $(QA)$, and $(NA)$ in the Beurling case or $(NE)$ in the Roumieu case. 

Next, we discuss the notion of support for quasianalytic functionals. For a compact set $K$ in $\R^d$, we define the space of germs of ultradifferentiable functions on $K$ as
$$
\mathcal{E}^\ast[K] = \varinjlim_{K \Subset \Omega} \mathcal{E}^{\ast}(\Omega).
$$
The elements of the dual spaces $\mathcal{E}'^\ast[K]$  are called \emph{local quasianalytic functionals of class $(M_p)$ or Beurling type (of class $\{M_p\}$ or Roumieu type) on} $K$. 
Let $\mathcal{N}(K)$ be a fundamental system of open neighborhoods of $K$.
Clearly, 
$$ 
\mathcal{E}^\ast[K] \cong \varinjlim_{\Omega \in \mathcal{N}(K)} \mathcal{E}^{\ast}(\Omega)
$$
as locally convex spaces. Notice that $\mathcal{E}^{(M_p)}[K]$ is a $(LFN)$-space while $\mathcal{E}^{\{M_p\}}[K]$ is a $(DFN)$-space, as follows from \cite[Thm.\ 2.6]{Komatsu}. Moreover, since
\begin{equation}
\label{projgermseq}
\mathcal{E}^\ast(\Omega) \cong \varprojlim_{K \Subset \Omega} \mathcal{E}^\ast[K]
\end{equation}
as l.c.s.\ for $\Omega$ open,  and $\mathcal{E}^\ast(\Omega)$ is dense in each $\mathcal{E}^\ast[K]$, we have the isomorphism of vector spaces
$$
\mathcal{E}'^\ast(\Omega) \cong \varinjlim_{K \Subset \Omega} \mathcal{E}'^\ast[K].
$$
If $\ast=\{M_{p}\}$, the isomorphism is in fact topological because each $\mathcal{E}^\ast[K]$ is reflexive.

Let $f \in \mathcal{E}'^\ast(\Omega)$, where $\Omega$ is open. A compact set $K \Subset \Omega$ is said to be a $\ast$-\emph{carrier of} 
$f$ if $f \in \mathcal{E}'^\ast[K]$. It is well known that for every $f \in \mathcal{A}'(\Omega)$ 
there is a smallest compact set $K \Subset \Omega$ among the $\{p!\}$-carriers of $f$, called the \emph{support of} $f$ and denoted by $\operatorname{supp}_{\mathcal{A}'}f$. This 
essentially follows from the cohomology of the sheaf of germs of analytic functions (see e.g. \cite{Morimoto}). An elementary proof based on the properties of the Poisson 
transform of analytic functionals is provided in \cite[Sect.\ 9.1]{Hormander2}. See \cite{Matsuzawa} for a proof by means of the heat kernel method.  H\"ormander noticed that a 
similar result holds for quasianalytic functionals of Roumieu type \cite[Cor.\ 3.5]{Hormander}. More precisely, he showed that for every $f \in \mathcal{E}'^{\{M_p\}}(\Omega)$ there is a smallest compact set among the $\{M_p\}$-carriers of $f$ and that this set coincides with $\operatorname{supp}_{\mathcal{A}'}f$. The corresponding statement for the Beurling case was shown in \cite[Thm.\ 4.11]{Heinrich} \footnote{The authors work there with the notion of ultradifferentiability defined via weight functions as in \cite{Braun}, but their proofs can also be adapted to the present setting.}. For future reference, we collect these facts in the following proposition.

\begin{proposition}[\cite{Heinrich,Hormander}]\label{support} Let $\Omega \subseteq \R^d$ be open.
 For every $f \in \mathcal{E}'^{*}(\Omega)$ the set $\operatorname{supp}_{\mathcal{A}'} f$ is the smallest compact set of $\Omega$ among the $\ast$-carriers of $f$. 
 \end{proposition}

It follows from \cite[Thm.\ 2.6]{Komatsu} that $\mathcal{E}^{(M_p)}(\Omega)$ is a $(FN)$-space, while $\mathcal{E}^{\{M_p\}}(\Omega)$ is a $(PLN)$-space, as follows from the projective representation \eqref{projgermseq}. In the next proposition we establish a topological property of $\mathcal{E}^{\{M_p\}}(\Omega)$ that shall be crucial for the rest of this work.

\begin{proposition}\label{topologytestspaces} 
Let $\Omega \subseteq \R^d$ be open. The space $\mathcal{E}^{\{M_p\}}(\Omega)$ is an ultrabornological $(PLN)$-space.
\end{proposition} 
\begin{proof}
Let $(K_n)_{n \in \N}$ be an exhaustion by compact subsets of $\Omega$. The projective spectrum $\mathcal{X} = (\mathcal{E}^{\{M_p\}} [K_n])_{n \in \N}$ (with canonical linking mappings) consists of $(DFN)$-spaces and is reduced. Moreover, we have the following isomorphism of l.c.s.
$$
\mathcal{E}^{\{M_p\}}(\Omega) \cong \varprojlim_{n \in \N} \mathcal{E}^{\{M_p\}}[K_n],
$$
which gives that $\mathcal{E}^{\{M_p\}}(\Omega)$ is a $(PLN)$-space as already mentioned above. By Proposition \ref{VogtandWeng}, it suffices to show that $\mathcal{X^\ast}$ is $\alpha$-regular. We have 
$$
\mathcal{E}'^{\{M_p\}}(\Omega) \cong \varinjlim_{n \in \N} \mathcal{E}^{'\{M_p\}}[K_n]
$$
as locally convex spaces. Let $B \subset \mathcal{E}'^{\{M_p\}}(\Omega)$ be bounded. A classical result of Martineau shows that $\mathcal{A}(\Omega)$ is an ultrabornological $(PLN)$-space \cite[Thm.\ 1.2, Prop.\ 1.9]{Martineau2}. Since the inclusion mapping $\mathcal{E}'^{\{M_p\}}(\Omega) \rightarrow \mathcal{A}'(\Omega)$ is continuous, Proposition \ref{VogtandWeng} implies that $B \subset \mathcal{A}'[K_n]$ for some $n \in \N$. The result now follows from Proposition \ref{support}.
\end{proof}
\begin{remark}\label{rkCousin1}
For $\Omega$ convex, Proposition \ref{topologytestspaces} is due to R\"osner \cite{Rosner}. To the best of our knowledge, the result was not yet known for general $\Omega$ . 
\end{remark}
\section{Structure theorem for quasianalytic functionals}\label{sect-structure}
The purpose of this section is to generalize Komatsu's first structure theorem \cite[Thm.\ 8.1]{Komatsu} for non-quasianalytic ultradistributions to quasianalytic functionals. As an application, we shall give an explicit system of seminorms generating the topology of the space $\mathcal{E}^{\{M_p\}}(\Omega)$ (cf.\ \cite[Prop.\ 3.5]{Komatsu3}). The latter result is indispensable for the treatment of vector-valued quasianalytic functions of Roumieu type in the next section. The analysis of the Beurling case is similar to that given in \cite{Komatsu}, but we include details for the sake of completeness. The Roumieu case requires more elaborate arguments.
\begin{proposition}\label{structureBeurling} Let $\Omega \subseteq \R^d$ be open. For every bounded set $B$ in $\mathcal{E}'^{(M_p)}(\Omega)$ there are a compact set $K \Subset \Omega$ and measures $\mu_\alpha(f) \in C'(K)$, $\alpha \in \N^d$, $f \in B$, such that
$$
\sup_{f \in B} \sup_{\alpha \in \N ^d} \frac{\| \mu_\alpha(f)\|_{C'(K)} M_\alpha}{h^{|\alpha|}} < \infty
$$
for some $h > 0$ and 
$$
f = \sum_{\alpha \in \N^d} (\mu_\alpha(f))^{(\alpha)}, \qquad f \in B.
$$
\end{proposition}
\begin{proof} Let $(K_n)_{n \in \N}$ be an exhaustion by regular compact subsets of $\Omega$. We write $X_n$ for the space of all $\varphi \in C^\infty(K_n)$ such that 
$$
\sup_{x \in K_n}\frac{n^{|\alpha|}|\varphi^{(\alpha)}(x)|}{M_\alpha} \to 0, \qquad \mbox{as } |\alpha| \to \infty;
$$
endowed with the norm $\| \, \|_{K_n, 1/n}$ it becomes a Banach space. Clearly, $\mathcal{E}^{(M_p)}(\Omega) = \varprojlim X_n$ as locally convex spaces. Since $\mathcal{E}^{(M_p)}(\Omega)$ is a Fr\'echet space, the set $B$ is equicontinuous and, by the Hahn-Banach theorem, it can be extended to an equicontinuous set $\widetilde{B}$ in $X'_n$ for some $n \in \N$. Set $K = K_n$ and define $Y_n$ as the space of all tuples $(\varphi_\alpha)_\alpha \in C(K)^{\N^d}$ such that 
$$
\frac{n^{|\alpha|}\|\varphi_\alpha\|_{C(K)}}{M_\alpha} \to 0, \qquad \mbox{as } |\alpha| \to \infty;
$$
endowed with the norm 
$$
\sup_{\alpha \in \N^d}\frac{n^{|\alpha|}\|\varphi_\alpha\|_{C(K)}}{M_\alpha}
$$
it becomes a Banach space. The mapping $\iota_1: X_n \rightarrow Y_n: \varphi \rightarrow (\varphi^{(\alpha)})_\alpha$ is an injective linear topological homomorphism. Next, write $U$ for the disjoint union of $\N^d
$ copies of $K$. It becomes a locally compact space with the disjoint union topology. Notice that $C_0(U)$ can be topologically identified with the Banach space of all tuples $(\varphi_\alpha)_\alpha \in C(K)^{\N^d}$ such that 
$$
\|\varphi_\alpha\|_{C(K)}\to 0, \qquad \mbox{as } |\alpha| \to \infty,
$$
endowed with the norm 
$$
\sup_{\alpha \in \N^d} \|\varphi_\alpha\|_{C(K)}.
$$
Hence the mapping $\iota_2: Y_n \rightarrow C_0(U): (\varphi_\alpha)_\alpha \rightarrow (n^{|\alpha|}\varphi_\alpha/ M_\alpha)_\alpha$ is an injective linear topological homomorphism. We set $\iota = \iota_2 \circ \iota_1: X_n \to C_0(U)$ and write $\rho: \iota(X_n) \rightarrow X_n$ for the continuous linear mapping such that $\rho \circ \iota = \operatorname{id}$. The Hahn-Banach theorem implies that the set $\{ g \circ \rho \, : \, g \in \widetilde{B}\}$ can be extended to an equicontinuous subset of $C'_0(U)$. By the Riesz representation theorem there are Borel measures $\widetilde{\mu}_\alpha(g) \in C'(K)$, $\alpha \in \N^d$, $g \in \widetilde{B}$, such that
$$
\sup_{g \in \widetilde{B}}\sum_{\alpha \in \N^d} \|\widetilde{\mu}_\alpha(g)\|_{C'(K)} < \infty,
$$
and 
$$
g = \sum_{\alpha \in \N^d} \widetilde{\mu}_{\alpha}(g)\circ \iota, \qquad g \in \widetilde{B}.
$$
Denote by $\widetilde{f}$ the extension of $f \in B$ to $\widetilde{B}$. It is then clear that the measures $\mu_\alpha(f) = (-n)^{|\alpha|} \widetilde{\mu}_\alpha(\widetilde{f})/M_\alpha$ satisfy all requirements.
\end{proof}

Our strategy to deal with the structure of Roumieu quasianalytic functionals is to reduce the problem to the Beurling case. For it, we employ the Fourier-Laplace transform and a support splitting theorem due to H\"ormander \cite[Thm.\ 5.1]{Hormander}. We need some preparation. 

The supporting function \cite[Sect.\ 1.8]{Morimoto} of a convex compact set $K$ of $\R^d$ is defined as
$$
h_K(\xi) := \sup_{x \in K} \operatorname{Re} (\xi \cdot x), \qquad \xi \in \C^d.
$$
For $\lambda > 0$ we write $\mathcal{O}^{M_p, \lambda}_K$ for the Banach space of all entire functions $F \in \mathcal{O}(\C^d)$ such that 
$$
\sup_{\xi \in \C^d} |F(\xi)|e^{-h_K(\xi) - M(\xi/\lambda)} < \infty.
$$
Given a convex open set $\Omega$ in $\R^d$, we define
$$
\mathcal{O}^{(M_p)}_{\Omega} = \varinjlim_{K \Subset \Omega} \varinjlim_{\lambda \rightarrow 0^+}\mathcal{O}^{M_p,\lambda}_{K}, \qquad  \mathcal{O}^{\{M_p\}}_{\Omega} = \varinjlim_{K \Subset  \Omega} \varprojlim_{\lambda \rightarrow \infty}\mathcal{O}^{M_p,\lambda}_{K}.
$$
Let $f \in \mathcal{E}'^{\ast}(\Omega)$, its Fourier-Laplace transform is defined as 
$$
\mathcal{F}(f)(\xi) = \widehat{f}(\xi) = \langle f(x), e^{-i\xi \cdot x} \rangle, \qquad \xi \in \C^d.
$$
It is known that $\mathcal{F}: \mathcal{E}'^{\ast}(\Omega) \rightarrow \mathcal{O}^{\ast}_{\Omega}$ is a linear topological isomorphism (see e.g.\ \cite{Heinrich, Takiguchi}). 

Next, we discuss H\"ormander's splitting theorem. We give a short 
proof using Proposition \ref{support}.  Let $K_1$ and 
$K_2$ be compact sets in $\R^d$ with $K_1 \subseteq K_2$. We write $\iota_{K_1,K_2}: \mathcal{E}^\ast[K_2] \to \mathcal{E}^\ast[K_1]$ for the canonical restriction mapping. Its transpose is the canonical inclusion mapping $\mathcal{E}'^\ast[K_1] \to \mathcal{E}'^\ast[K_2]$. We shall identify $f \in  \mathcal{E}'^\ast[K_1]$ with its image under the  mapping $^t\iota_{K_1,K_2}$.
\begin{proposition}\label{splitting}
Let $K_1$ and $K_2$ be compact sets in $\R^d$. The sequence
\begin{center}
\begin{tikzpicture}
  \matrix (m) [matrix of math nodes, row sep=2em, column sep=2em]
    {0 & \mathcal{E}'^{\{M_p\}}[K_1 \cap K_2]& \mathcal{E}'^{\{M_p\}}[K_1] \times \mathcal{E}'^{\{M_p\}}[K_2]  & \mathcal{E}'^{\{M_p\}}[K_1 \cup K_2] & 0 \\
   };
  
  { [start chain] \chainin (m-1-1);
\chainin (m-1-2);
\chainin (m-1-3) [join={node[above,labeled] {S}}];
\chainin (m-1-4)[join={node[above,labeled] {T}}];
\chainin (m-1-5)[join={node[above,labeled] {}}];
}
  
\end{tikzpicture}
\end{center}
is topologically exact, where $S(f) = (f,f)$ and $T(f_1,f_2) = f_2-f_1$. Moreover, for every bounded set $B \subset \mathcal{E}'^{\{M_p\}}[K_1 \cup K_2]$ there are bounded sets $B_j \subset  \mathcal{E}'^{\{M_p\}}[K_j]$, $j = 1,2$, such that  
$T(B_1,B_2) = B$.
\end{proposition}
\begin{proof}
In view of the open mapping theorem, it suffices to show that the sequence is algebraically exact. The injectivity of $S$ is clear while the equality $\operatorname{Im} S = \operatorname{ker } T$ follows from Proposition \ref{support}. Notice that the transpose of $T$ may be identified with the mapping
$$
\mathcal{E}^{\{M_p\}}[K_1 \cup K_2] \rightarrow \mathcal{E}^{\{M_p\}}[K_1] \times \mathcal{E}^{\{M_p\}}[K_2]: \varphi \rightarrow (-\iota_{K_1,K_1 \cup K_2}(\varphi),\iota_{K_2,K_1 \cup K_2}(\varphi)).
$$
The mapping $T$ is surjective since the above mapping is injective and has closed range. Indeed, the injectivity is clear while the closed range property follows from the fact that the range is equal to the kernel of the  continuous mapping
$$
\mathcal{E}^{\{M_p\}}[K_1] \times \mathcal{E}^{\{M_p\}}[K_2] \rightarrow \mathcal{E}^{\{M_p\}}[K_1 \cap K_2]: (\varphi_1, \varphi_2) \rightarrow \iota_{K_1 \cap K_2, K_1 }(\varphi_1) + \iota_{K_1 \cap K_2, K_2}(\varphi_2).
$$
The second part follows from the general fact that for an exact sequence of Fr\'echet spaces 
$$
0 \longrightarrow X \longrightarrow Y \xrightarrow{\phantom,T\phantom,} Z \longrightarrow 0
$$
with $X$ an $(FS)$-space it holds that for every bounded set $B \subset Z$ there is a bounded set $A \subset Y$ such that $T(A) = B$ \cite[Lemma 26.13]{Meise}.
\end{proof}
We write $\mathcal{R}$ for the family of positive real sequences $(r_j)_{j \in \N}$ with $r_0 =1$ which increase (not necessarily strictly) to infinity. This set is partially ordered and directed by the relation $r_j \preceq s_j$, which means that there is a $j_0 \in \N$ such that $r_j \leq s_j$ for all $j \geq j_0$. Let $M_p$ be a weight sequence with associated function $M$ and let $r_j \in \mathcal{R}$. We denote by $M_{r_j}$ the associated function of the sequence $M_p\prod_{j = 0}^pr_j$. We need three technical lemmas.
\begin{lemma}\label{bettersequence} Let $M_p$ be a weight sequence satisfying $(M.1)$, $(M.2)'$, and $(QA)$. Then, for every $r_j \in \mathcal{R}$ there is $r'_j \in \mathcal{R}$ with $r'_j \leq r_j$, $j \in \N$, such that the sequence $M_p\prod_{j = 0}^pr'_j$ also satisfies  $(M.1)$, $(M.2)'$, and $(QA)$.
\end{lemma}
\begin{proof}
Set $k_0 = 1$ and
$$
k_j = 1+ \left( \sum_{p = 1}^j \frac{1}{m_{p}} \right)^{1/2}, \qquad j \geq 1.
$$
The sequence $r'_j \in \mathcal{R}$ with $r'_0 = 1$ and 
$$
r'_j = \min( r_j, 2^j, k_j), \qquad j \geq 1,
$$
satisfies all requirements.
\end{proof}
\begin{lemma}\cite[Lemma 4.5]{DVV}\label{projective}
Let $M_p$ be a weight sequence satisfying $(M.1)$ and $(M.2)'$ and let $g: [0,\infty) \rightarrow [0,\infty)$. Then, $g(t) = O(e^{M(t/\lambda)})$ for all $\lambda > 0$ if and only if $g(t) = O(e^{M_{r_j}(t)})$ for some $r_j \in \mathcal{R}$. 
\end{lemma}
\begin{lemma}\label{reduction}
Let $\Omega \subseteq \R^d$ be open. For every bounded set $B$ in $\mathcal{E}'^{\{M_p\}}(\Omega)$ there is a weight sequence $N_p$ with $M_p \prec N_p$ satisfying $(M.1)$, $(M.2)'$, and $(QA)$ such that $B$ is contained and bounded in $\mathcal{E}'^{(N_p)}(\Omega)$. 
\end{lemma}
\begin{proof} STEP $1$: $\Omega$ is convex. From the above remarks on the Fourier-Laplace transform it follows that there is a convex compact set $K \subseteq \Omega$ such that 
$$
\sup_{f \in B}\sup_{\xi \in \C^d} |\widehat{f}(\xi)|e^{-h_K(\xi) - M(\xi/\lambda)} < \infty,
$$ 
for all $\lambda > 0$. Applying Lemma \ref{projective} to the function 
$$
g(t) = \sup_{f \in B}\sup_{|\xi| = t} |\widehat{f}(\xi)|e^{-h_K(\xi)}, \qquad t \geq  0,
$$
we find a sequence $r_j \in \mathcal{R}$ such that
$$
\sup_{f \in B}\sup_{\xi \in \C^d} |\widehat{f}(\xi)|e^{-h_K(\xi) - M_{r_j}(\xi)} < \infty.
$$
By Lemma \ref{bettersequence}, we may assume without loss of generality that the sequence $N_p :=M_p\prod_{j = 0}^pr_j$ satisfies $(M.1)$, $(M.2)'$, and $(QA)$. The result now follows from the fact that  $\mathcal{F}: \mathcal{E}'^{(N_p)}(\Omega) \rightarrow \mathcal{O}^{(N_p)}_{\Omega}$ is a topological isomorphism.

STEP $2$: $\Omega$ is arbitrary. By Propositions \ref{VogtandWeng} and \ref{topologytestspaces} there is $K \Subset \Omega$ such that $B$ is contained and bounded in $\mathcal{E}'^{\{M_p\}}[K]$. Let $K_1, \ldots, K_N$, $N \in \N$, be convex compact sets in $\Omega$ such that $K \subseteq \bigcup_{j = 1}^N K_j$. Using Proposition \ref{splitting} and a simple induction argument, we find bounded sets $B_j \subset \mathcal{E}'^{\{M_p\}}[K_j]$, $j = 1, \ldots N$, such that $B = B_1 + \cdots + B_N$.  The result now follows from the first step and Lemma \ref{bettersequence}.
\end{proof}
\begin{remark}
The technique of reducing the case of arbitrary open sets to open convex sets as in Lemma \ref{reduction} is due to Heinrich and Meise \cite{Heinrich}.
\end{remark}
Propositions \ref{structureBeurling} and Lemma \ref{reduction} immediately yield the analogue of Komatsu's first structure theorem in the Roumieu case.
\begin{proposition}\label{structureRoumieu} Let $\Omega \subseteq \R^d$ be open. For every bounded set $B$ in $\mathcal{E}'^{\{M_p\}}(\Omega)$ there are a compact set $K \Subset \Omega$ and measures $\mu_\alpha(f) \in C'(K)$, $\alpha \in \N^d$, $f \in B$, such that
$$
\sup_{f \in B} \sup_{\alpha \in \N ^d} \frac{\| \mu_\alpha(f)\|_{C'(K)} M_\alpha}{h^{|\alpha|}} < \infty
$$
for all $h > 0$ and 
$$
f = \sum_{\alpha \in \N^d} (\mu_\alpha(f))^{(\alpha)}, \qquad f \in B.
$$
\end{proposition}

\begin{proposition}\label{Roumieuseminorms} Let $\Omega \subseteq \R^d$ be open. A function $\varphi \in C^\infty(\Omega)$ belongs to $\mathcal{E}^{\{M_p\}}(\Omega)$ if and only if 
$$
\| \varphi \|_{K, r_j} := \sup_{\alpha \in \N^d}\sup_{x \in K} \frac{|\varphi^{(\alpha)}(x)|}{M_{\alpha}\prod_{j = 0}^{|\alpha|}r_j} < \infty
$$
for all $K \Subset \Omega$ and $r_j \in \mathcal{R}$. Moreover, the topology of $\mathcal{E}^{\{M_p\}}(\Omega)$ is generated by the system of seminorms $\| \, \|_{K, r_j}$.
\end{proposition}
\begin{proof}
The first part follows from \cite[Lemma 3.4]{Komatsu3}, we thus only have to check the topological assertion.  Clearly, every seminorm $\| \, \|_{K, r_j}$ acts continuously on $\mathcal{E}^{\{M_p\}}(\Omega)$. Conversely, let $p$ be a continuous seminorm on $\mathcal{E}^{\{M_p\}}(\Omega)$.
There is a bounded set $B \subset \mathcal{E}'^{\{M_p\}}(\Omega)$ such that 
$$
p(\varphi) \leq \sup_{f \in B} |\langle f, \varphi \rangle|, \qquad \varphi \in \mathcal{E}^{\{M_p\}}(\Omega).
$$
Proposition \ref{structureRoumieu} implies that there are a compact set $K \Subset \Omega$ and measures $\mu_\alpha(f) \in C'(K)$, $\alpha \in \N^d$, $f \in B$, such that
$$
\sup_{f \in B} \sup_{\alpha \in \N ^d} \frac{\| \mu_\alpha(f)\|_{C'(K)} M_\alpha}{h^{|\alpha|}} < \infty
$$
for all $h > 0$ and 
$$
f = \sum_{\alpha \in \N^d} (\mu_\alpha(f))^{(\alpha)}, \qquad f \in B.
$$
Hence,
$$
\sup_{f \in B} |\langle f, \varphi \rangle| \leq \sup_{f \in B} \sum_{\alpha \in \N^d}  \|\mu_\alpha(f)\|_{C'(K)} \|\varphi^{(\alpha)}\|_{C(K)} , 
$$
and the result follows once again from \cite[Lemma 3.4]{Komatsu3}.
\end{proof}
\section{Vector-valued quasianalytic functions}\label{sect-vector}
We now turn our attention to spaces of vector-valued quasianalytic functions and their topological properties. Our first goal is to derive a tensor product representation of these spaces.

Given two l.c.s.\ $X$ and $Y$ we denote by $L(X, Y)$ the space of all continuous linear mappings from $X$ into $Y$.  We write $L_\beta(X,Y)$ ($L_c(X,Y)$, resp.) if we want to indicate that we endow $L(X,Y)$ with the strong topology (topology of uniform convergence on balanced convex compact sets). We use the same notation for indicating the topology on $X'$. Recall that if we merely write $X'$ we implicitly endow it with the strong topology. For $t = \beta, c$ we denote by $L_\varepsilon(X'_t, Y )$ the space $L(X'_t, Y )$ endowed with the topology of uniform convergence on equicontinuous subsets of $X'$.

Following Schwartz \cite{Schwartz} and Komatsu \cite{Komatsu3},  we denote by $X\varepsilon Y$ (the $\varepsilon$-product of $X$ and $Y$) the space of all bilinear functionals on $X'_c \times Y'_c$ which are hypocontinuous with respect to the equicontinuous subsets of $X'$ and $Y'$. We endow it with the topology of uniform convergence on products of equicontinuous subsets of $X'$ and $Y'$. As pointed out in \cite[p.\ 657]{Komatsu3}, we have the following canonical isomorphisms of l.c.s.
$$
X \varepsilon Y \cong L_\varepsilon(X'_c,Y) \cong L_\varepsilon(Y'_c,X).
$$

The tensor product $X \otimes Y$ is canonically embedded into $X\varepsilon Y$ via $(x\otimes y)(x', y') = \langle x', x\rangle \langle y', y\rangle$. Clearly, the induced topology on $X \otimes Y$ is the $\varepsilon$-topology. Given continuous linear mappings $T_1: X_1 \rightarrow Y_1$ and $T_2: X_2 \rightarrow Y_2$ we write $T_1 \varepsilon T_2: X_1 \varepsilon X_2 \rightarrow Y_1\varepsilon Y_2$ for the continuous linear mapping given by
$$
T_1 \varepsilon T_2(\Phi)(y'_1,y'_2) = \Phi({}^tT_1 y'_1, {}^tT_2 y'_2), \qquad y'_j \in Y'_j, j = 1,2.
$$
The restriction of $T_1 \varepsilon T_2$ to $X_1 \otimes X_2$ is equal to the tensor product of the mappings $T_1$ and $T_2$.

If $X$ and $Y$ are complete and if either $X$ or $Y$ has the weak approximation property, in particular, if either $X$ or $Y$ is nuclear, we have $X \varepsilon Y = X \widehat{\otimes}_\varepsilon Y$ as locally convex spaces \cite[Prop.\ 1.4]{Komatsu3}. As usual, if either $X$ or $Y$ is nuclear, we write $X \widehat{\otimes}Y :=  X \widehat{\otimes}_\varepsilon Y =  X \widehat{\otimes}_\pi Y$.

We now introduce spaces of vector-valued quasianalytic functions. Let $\Omega \subseteq \R^d$ be open and let $F$ be a locally convex space. We write $\mathcal{E}^{(M_p)}(\Omega; F)$ ($\mathcal{E}^{\{M_p\}}(\Omega; F)$) for the space of all $\bm{\varphi} \in C^\infty(\Omega ; F)$ such that, for each  continuous seminorm $q$ on $F$, $K \Subset \Omega$, and $h > 0$ ($r_j \in \mathcal{R}$),
$$
q_{K,h}(\bm{\varphi}) := \sup_{\alpha \in \N^d}\sup_{x \in K} \frac{q(\bm{\varphi}^{(\alpha)}(x))}{h^{|\alpha|}M_{\alpha}} < \infty \qquad \left( q_{K,r_j}(\bm{\varphi}) := \sup_{\alpha \in \N^d}\sup_{x \in K} \frac{q(\bm{\varphi}^{(\alpha)}(x))}{M_{\alpha}\prod_{j = 0}^{|\alpha|}r_j} < \infty \right). 
$$
We endow it with the locally convex topology generated by the system of seminorms $q_{K,h}$ ($q_{K,r_j}$). Notice that, in view of \cite[Lemma 3.4]{Komatsu3}, $\mathcal{E}^{\{M_p\}}(\Omega; F)$ coincides with the set of all $\bm{\varphi} \in C^\infty(\Omega ; F)$ such that for each  continuous seminorm $q$ on $F$ and $K \Subset \Omega$ there is $h > 0$ such that $q_{K,h}(\bm{\varphi}) < \infty$.

\begin{proposition}\label{charvectorvaluedfunctions}
Let $\Omega \subseteq \R^d$ be open and let $F$ be a sequentially complete locally convex space. Then, $\mathcal{E}^{\ast}(\Omega; F)$ coincides with the space of all functions $\bm{\varphi}: \Omega \rightarrow F$ such that 
$\langle y', \bm{\varphi}(\cdot) \rangle \in \mathcal{E}^\ast(\Omega)$ for all $y' \in F'$. Moreover, we have the following canonical isomorphism of l.c.s.
$$
\mathcal{E}^{\ast}(\Omega; F) \cong \mathcal{E}^{\ast}(\Omega) \varepsilon F,
$$
and, if $F$ is complete,
$$
\mathcal{E}^{\ast}(\Omega; F) \cong \mathcal{E}^{\ast}(\Omega) \varepsilon F \cong \mathcal{E}^{\ast}(\Omega) \widehat{\otimes} F. 
$$
\end{proposition}
The proof of Proposition \ref{charvectorvaluedfunctions} is based on the following criterium due to Komatsu.
\begin{lemma}\cite[Lemma 1.12]{Komatsu3} \label{criteriumKomatsu} Let $\Omega$ be a $\sigma$-compact metrizable locally compact space, let $X$ be a space consisting of continuous scalar-valued functions on $\Omega$ equipped with a  locally convex topology that is semi-Montel and stronger than the topology of uniform convergence on compact subsets of $\Omega$, and let $F$ be a sequentially complete locally convex space.  Suppose that the sequential closure in $X'_c$ of the set of functionals represented by measures with compact support in $\Omega$ is equal to $X'$. Then, $X\varepsilon F \cong L(F'_c, X)$ may be identified with the space of all $\bm{\varphi}: \Omega \rightarrow F$ such that $\langle y', \bm{\varphi}(\cdot) \rangle \in X$ for all $y' \in F'$.
\end{lemma}

\begin{lemma}\label{densitymeasures}
Let $\Omega \subseteq \R^d$ be open. The sequential closure in $\mathcal{E}_\beta'^\ast(\Omega)$ of the linear span of the set $\{ \delta_x  : \, x \in \Omega\}$ is equal to $\mathcal{E}'^\ast(\Omega)$.
\end{lemma}
\begin{proof}
\emph{Beurling case}: The sequential closure of a subset in a $(DFS)$-space is equal to its closure (cf.  \cite[ Prop.\ 8.5.28]{Bonet2}). Therefore, it suffices to show that the linear span of the set $\{ \delta_x : \, x \in \Omega\}$ is dense in  $\mathcal{E}_\beta'^{(M_p)}(\Omega)$, but this follows at once from the Hahn-Banach theorem and the fact that the space $\mathcal{E}^{(M_p)}(\Omega)$ is reflexive.

$\emph{Roumieu case}$: Let $(\Omega_n)_{n\in \N}$ be an exhaustion by relatively compact open subsets of $\Omega$. We have  $\mathcal{E}_\beta'^{\{M_p\}}(\Omega) = \varinjlim \mathcal{E}_\beta'^{\{M_p\}}[\overline{\Omega}_n]$ as locally convex spaces. Let $n \in \N$ be arbitrary. The condition $(QA)$ implies that an element $\varphi \in \mathcal{E}^{\{M_p\}}[\overline{\Omega}_n]$ is equal to zero if and only if  one (and hence all) of its representatives vanishes on $\Omega_n$. Hence, by the Hahn-Banach theorem and the fact that the space $\mathcal{E}^{\{M_p\}}[\overline{\Omega}_n]$ is reflexive, we obtain that the linear span of the set $\{ \delta_x \, : \, x \in \Omega_n\}$ is dense in $\mathcal{E}_\beta'^{\{M_p\}}[\overline{\Omega}_n]$. Since the latter space is a Fr\'echet space, we actually have that for each $f \in \mathcal{E}'^{\{M_p\}}[\overline{\Omega}_n]$ there is a sequence $(f_j)_{j \in \N} \subset \operatorname{span} \{ \delta_x \, : \, x \in \Omega_n\} \subset  \operatorname{span} \{ \delta_x \, : \, x \in \Omega\}$ such that $f_j \to f$, as $j \to \infty$, in  $\mathcal{E}_\beta'^{\{M_p\}}[\overline{\Omega}_n]$ and, thus, in $\mathcal{E}_\beta'^{\{M_p\}}(\Omega)$. \end{proof}

\begin{proof}[Proof of Proposition \ref{charvectorvaluedfunctions}] With the aid of Lemma \ref{densitymeasures},  the proof now becomes identical to that of \cite[Thm.\ 3.10]{Komatsu3}.  We repeat the argument for the sake of completeness.
We only show the Roumieu case, the Beurling case is similar. Clearly, $\bm{\varphi} \in \mathcal{E}^{\{M_p\}}(\Omega ; F)$ implies that $\langle y', \bm{\varphi}(\cdot) \rangle \in \mathcal{E}^{\{M_p\}}(\Omega)$ for all $y' \in F'$. Conversely, let $\bm{\varphi} : \Omega \rightarrow F$ be a function having the latter property. In particular, it holds that $\langle y', \bm{\varphi}(\cdot) \rangle \in C^\infty(\Omega)$ for all $y' \in F'$ and, thus, by a well known result, that $\bm{\varphi} \in C^\infty(\Omega; F)$ and 
\begin{equation}
\langle y', \bm{\varphi}(\cdot) \rangle^{(\alpha)} = \langle y', \bm{\varphi}^{(\alpha)}(\cdot) \rangle, \qquad y' \in F', \alpha \in \N^d.
\label{changeder}
\end{equation}
Hence, by Proposition \ref{Roumieuseminorms}, we obtain that for each $K \Subset \Omega$ and $r_j \in \mathcal{R}$  the set
$$
\left \{  \frac{\bm{\varphi}^{(\alpha)}(x)}{M_\alpha \prod_{j = 0}^{|\alpha|}r_j} \, : \, x \in K, \alpha \in \N^d    \right\}
$$
is weakly bounded in $F$. By Mackey's theorem the set is bounded in $F$, which precisely means that $\bm{\varphi} \in \mathcal{E}^{\{M_p\}}(\Omega ; F)$. This shows the first part of the proposition. By Lemmas \ref{criteriumKomatsu} and \ref{densitymeasures}, we therefore have $\mathcal{E}^{\{M_p\}}(\Omega; F) \cong \mathcal{E}^{\{M_p\}}(\Omega) \varepsilon F$ as vector spaces. We now show that this isomorphism also holds topologically. Let $K \Subset \Omega$, $r_j \in \mathcal{R}$, and let $q$ be an arbitrary continuous seminorm on $F$.  Define $A$ to be the polar set of the $\| \, \|_{K, r_j}$-unit ball in $\mathcal{E}^{\{M_p\}}(\Omega)$ and $B$ to be the polar set of the $q$-unit ball in $F$. Hence, by \eqref{changeder} and the bipolar theorem,
\begin{align*}
\sup \{ | \langle f , \langle y', \bm{\varphi}(\cdot)\rangle \rangle | \, : \, f \in A, y' \in B\} &= \sup \{ \| \langle y', \bm{\varphi}(\cdot)\rangle \|_{K,r_j} \, : \, y' \in B\}  \\
&= \sup \left \{  \frac{|\langle y',\bm{\varphi}^{(\alpha)}(x) \rangle|}{M_\alpha \prod_{j = 0}^{|\alpha|}r_j} \, : \,  y' \in B, x \in K, \alpha \in \N^d    \right\}  \\
&= q_{K,r_j}(\bm{\varphi}).
\end{align*} 
In view of Proposition \ref{Roumieuseminorms}, this shows that the above isomorphism indeed holds topologically. The last part follows from the fact that the space $\mathcal{E}^{\{M_p\}}(\Omega)$ is nuclear.
\end{proof}

Next, we are interested in the topological properties of the spaces $\mathcal{E}^\ast(\Omega;F)$. We start with a discussion about the $\varepsilon$-product of $(PLS)$-spaces. Let $X = \varprojlim X_n$ be a $(PLN)$-space with $(X_n)_{n \in \N}$ a reduced projective spectrum of $(DFN)$-spaces and let $Y = \varprojlim Y_n$ be a $(PLS)$-space with $(Y_n)_{n \in \N}$ a reduced projective spectrum of $(DFS)$-spaces. First notice that, by \cite[Prop.\ 1.5]{Komatsu3}, we have the following isomorphism of l.c.s.
$$
X\varepsilon Y \cong \varprojlim_{n \in \N} {X_n \varepsilon Y_n}.
$$ 
Moreover, as the $\varepsilon$-product of two $(DFS)$-spaces is again a $(DFS)$-space \cite[Prop. 4.3]{Bierstedt} and $X\varepsilon Y = X \widehat{\otimes} Y$ is dense in each $X_n\varepsilon Y_n = X_n \widehat{\otimes} Y_n$, $X\varepsilon Y$ is a $(PLS)$-space which can be represented as the projective limit of the reduced spectrum $(X_n \varepsilon Y_n)_{n \in \N}$  of $(DFS)$-spaces. It is highly desirable to find conditions on $X$ and $Y$ which ensure that $X\varepsilon Y$ is ultrabornological. Doma\'nski \cite{Domanski2} achieved this by making use of the so called dual interpolation estimates for $(PLS)$-spaces. These were introduced in \cite{Bonet} and can be viewed as abstract Phragm\'en-Lindel\"of conditions. Let us discuss the precise definition. 

Let $X = \varprojlim X_n$ be a $(PLS)$-space with $(X_n)_{n \in \N}$ a projective spectrum of $(DFS)$-spaces. Suppose that the $X_n$ are given by 
$$X_n = \varinjlim_{N \in \N} X_{n,N}$$
with $(X_{n,N}, \| \, \|_{n,N})$ Banach spaces. We say that $X$ has the \emph{dual interpolation estimate for small theta} if 
$$
\forall n \; \exists m \geq n \; \forall k \geq m \; \exists N \; \forall M \geq N \; \exists \theta_0 \in (0, 1) \; \forall \theta \in (0, \theta_0)\; \exists K \
\geq M \; \exists C> 0 \; \forall x' \in X'_n :
$$
$$
\| x'\|^\ast_{m,M} \leq C \left( \| x'\|^\ast_{k,K}  \right)^{1-\theta}  \left( \|x'\|^\ast_{n,N}  \right)^{\theta}.
$$
It is known that the dual interpolation estimate for small theta implies that the space is ultrabornological \cite[Thm.\ 3.2.18]{Wengenroth2}. Moreover, by using \cite[Prop.\ 1.1]{Bonet}, one can readily check that a $(DFS)$-space $X$ satisfies the dual interpolation estimate for small theta if and only if $X'$ satisfies Vogt's condition $(\underline{DN})$ (see \cite[p.\ 368]{Meise} for the definition). The following proposition of Bonet and Doma\'nski is very important for us.

\begin{proposition}\cite[Cor.\ 2.2]{Bonet}\label{realanalyticDIE}
Let $\Omega \subseteq \R^d$ be open. The space $\mathcal{A}(\Omega)$ satisfies the dual interpolation estimate for small theta. 
\end{proposition}

\begin{proposition}\label{topologyvectorvalued}
Let $\Omega \subseteq \R^d$ be open.
\begin{itemize}
\item[$(i)$] If $F$ is a Fr\'echet space, then $\mathcal{E}^{(M_p)}(\Omega; F)$ is a Fr\'echet space.
\item[$(ii)$] If $F$ is a $(DFS)$-space such that $F'$ satisfies  $(\underline{DN})$, then $\mathcal{E}^{\{M_p\}}(\Omega; F)$ is an ultrabornological $(PLS)$-space.
\end{itemize}
\end{proposition}
The proof of Proposition \ref{topologyvectorvalued}$(i)$ is easy, one just has to combine Proposition \ref{charvectorvaluedfunctions} with the fact that the $\varepsilon$-product of two Fr\'echet spaces is again a Fr\'echet space.  For Proposition \ref{topologyvectorvalued}$(ii)$, we use the following result due to Doma\'nski.
\begin{proposition}\cite[Thm.\ 5.6]{Domanski2}  \label{epsilonPLS}
Let $X$ be a $(PLN)$-space and $F$ a $(PLS)$-space. Suppose that both $X$ and $F$ satisfy the dual interpolation estimate for small theta. Then, $X \varepsilon F$ is an ultrabornological $(PLS)$-space.\end{proposition}
\begin{remark}
Doma\'nski showed the above result under the additional assumption that the space $X$ is so called deeply reduced. By \cite[Prop.\ 8]{Piszczek} this assumption is superfluous. Moreover, based on results of Doma\'nski, Piszczek was able to show that the $(PLS)$-space $X \varepsilon F$ also satisfies the dual interpolation estimate for small theta \cite[Thm.\ 9]{Piszczek}.
\end{remark}
\begin{remark}\label{realanalytic-top}
The space $\mathcal{A}(\Omega ;F)$ is ultrabornological  for any $(PLS)$-space $F$ satisfiying the dual interpolation estimate for small theta, as immediately follows from Propositions \ref{charvectorvaluedfunctions}, \ref{realanalyticDIE}, and \ref{epsilonPLS}.
\end{remark}
\begin{proof}[Proof of Proposition \ref{topologyvectorvalued}(ii)]
We use the same technique as in Proposition \ref{topologytestspaces}. Therefore, we first give a representation of the dual of $\mathcal{E}^{\{M_p\}}(\Omega; F)$. Let $(K_n)_{n \in N}$ be an exhaustion by compact subsets of $\Omega$. By Proposition \ref{charvectorvaluedfunctions}, \cite[Prop.\ 1.5]{Komatsu3}, and \cite[Prop.\ 2.3]{Komatsu3}, we have the following isomorphisms of l.c.s.
$$
\mathcal{E}'^{\{M_p\}}(\Omega; F) \cong (\mathcal{E}^{\{M_p\}}(\Omega) \widehat{\otimes} F)' \cong (\varprojlim_{n\in \N} \mathcal{E}^{\{M_p\}}[K_n] \widehat{\otimes} F)' \cong \varinjlim_{n \in \N} L_\beta( \mathcal{E}^{\{M_p\}}[K_n], F').
$$
Let  $K \Subset \R^d$. Using the isomorphism $L(\mathcal{E}^{\{M_p\}}[K], F') \cong L(F, \mathcal{E}'^{\{M_p\}}[K])$ and the Pt\'ak closed graph theorem, one deduces
\begin{align} \label{subsupport}
&L(\mathcal{E}^{\{M_p\}}[K], F') \nonumber \\
&\cong \{ \mathbf{f} \in L(\mathcal{E}^{\{M_p\}}(\R^d), F') \, : \, y \circ \mathbf{f} \in \mathcal{E}'^{\{M_p\}}[K], \quad  \forall y \in F \},
\end{align}
as vector spaces. By Proposition \ref{VogtandWeng}, it suffices to show that the inductive spectrum 
$$
(L_\beta( \mathcal{E}^{\{M_p\}}[K_n], F'))_{n \in \N}
$$
 is $\alpha$-regular. Let $B \subset \mathcal{E}'^{\{M_p\}}(\Omega; F)$ be bounded. Since the canonical inclusion mapping $\mathcal{E}'^{\{M_p\}}(\Omega; F)  \rightarrow \mathcal{A}'(\Omega; F)$ is continuous, Remark \ref{realanalytic-top} implies that $B \subset L(\mathcal{A}[K_n], F')$ for some $n \in \N$. Hence, by \eqref{subsupport} and Proposition \ref{support}, we obtain that $B \subset  L( \mathcal{E}^{\{M_p\}}[K_n], F')$. 
\end{proof}
\section{The Cousin problem}\label{sect-Cousin}
We are ready to solve the Cousin problem for quasianalytic functions.
\subsection{Scalar-valued case}
It is natural to formulate the Cousin problem in the language of cohomology groups with coefficients in a sheaf. We therefore start with a brief discussion of the basic notions from this theory. For a detailed exposition, we refer to \cite[Chap.\ 4]{Morimoto}. 

Let $X$ be a topological space and let $\mathcal{F}$ be a sheaf on $X$. We denote  by $\Gamma(U, \mathcal{F})$ the sections of $\mathcal{F}$ on an open set $U$ of $X$.  Let $\mathcal{M} = \{ U_i \, : i \in I \}$ be a collection of open subsets of $X$. We write
$$
U_{i_0, \ldots, i_p} = U_{i_0} \cap \cdots \cap U_{i_p}, \qquad p \in \N, i_0, \ldots, i_p \in I.
$$
We define $C^p(\mathcal{M}, \mathcal{F})$, $p \in \N$, as the set consisting of collections $\varphi = (\varphi_{i_0, \ldots, i_p}) \in \prod_{(i_0, \ldots, i_p) \in I^{p+1}}  \Gamma(U_{i_0, \ldots, i_p}, \mathcal{F})$ which are antisymmetric with respect to the indices $i_0, \ldots, i_p$. For $\varphi \in C^p(\mathcal{M}, \mathcal{F})$ we define $\delta_p \varphi \in C^{p+1}(\mathcal{M}, \mathcal{F})$ as
$$
(\delta_p \varphi)_{i_0, \ldots, i_{p+1}} = \sum_{j = 0}^{p+1} (-1)^j \varphi_{i_0, \ldots, \widehat{i_j}, \ldots i_{p+1}| U_{i_0, \ldots, i_{p+1}}}, \qquad i_0, \ldots, i_{p+1} \in I,
$$
where, as usual, the hat mark on $\widehat{i_j}$ means that the index $i_j$ is omitted. Since $\delta_{p+1} \circ \delta_p = 0$, we have the complex
\begin{equation}
\begin{tikzpicture}
  \matrix (m) [matrix of math nodes, row sep=2em, column sep=2em]
    {0 & C^0(\mathcal{M}, \mathcal{F}) & C^1(\mathcal{M}, \mathcal{F})  & C^2(\mathcal{M}, \mathcal{F}) & \cdots. \\
   };
  
  { [start chain] \chainin (m-1-1);
\chainin (m-1-2);
\chainin (m-1-3) [join={node[above,labeled] {\delta_0}}];
\chainin (m-1-4)[join={node[above,labeled] {\delta_1}}];
\chainin (m-1-5)[join={node[above,labeled] {\delta_2}}];
}
\end{tikzpicture} 
\label{complexcovering}
\end{equation}
Define $Z^p(\mathcal{M}, \mathcal{F}) = \operatorname{ker} \delta_p$, $B^p(\mathcal{M}, \mathcal{F}) = \operatorname{Im} \delta_{p-1}$ ($B^0(\mathcal{M}, \mathcal{F}) = \{0\}$), and
$$
H^p(\mathcal{M}, \mathcal{F}) = Z^p(\mathcal{M}, \mathcal{F}) / B^p(\mathcal{M}, \mathcal{F}), \qquad p \in \N,
$$
that is, the $p$-th cohomology group of the complex \eqref{complexcovering}.

Let $I' \subseteq I$ and set $\mathcal{M'} = \{ U_i \, : i \in I' \}$. By restricting the indices of an element of $C^p(\mathcal{M}, \mathcal{F})$ to $I'$, we can naturally define the restriction mapping $C^p(\mathcal{M}, 
\mathcal{F}) \rightarrow C^{p}(\mathcal{M'}, \mathcal{F})$. We write $C^p(\mathcal{M}, \mathcal{M'}, \mathcal{F})$ for the kernel of this mapping and define $H^p(\mathcal{M}, \mathcal{M'}, 
\mathcal{F})$ to be the $p$-th cohomology group of the complex 
\begin{center}
\begin{tikzpicture}
  \matrix (m) [matrix of math nodes, row sep=2em, column sep=2em]
    {0 & C^0(\mathcal{M}, \mathcal{M'}, \mathcal{F}) & C^1(\mathcal{M},\mathcal{M'}, \mathcal{F})  & C^2(\mathcal{M}, \mathcal{M'}, \mathcal{F}) & \cdots. \\
   };
  
  { [start chain] \chainin (m-1-1);
\chainin (m-1-2);
\chainin (m-1-3) [join={node[above,labeled] {\delta_0}}];
\chainin (m-1-4)[join={node[above,labeled] {\delta_1}}];
\chainin (m-1-5)[join={node[above,labeled] {\delta_2}}];
}
  
\end{tikzpicture}
\end{center}
We have the following complex of short exact sequences
\begin{center}
\begin{tikzpicture}
  \matrix (m) [matrix of math nodes, row sep=2em, column sep=2em]
    {& 0 & 0 & 0 &  \\ 
    0 & C^0(\mathcal{M}, \mathcal{M'}, \mathcal{F}) & C^0(\mathcal{M}, \mathcal{F})  & C^0(\mathcal{M'}, \mathcal{F}) & 0 \\
     0 & C^1(\mathcal{M}, \mathcal{M'}, \mathcal{F}) & C^1(\mathcal{M}, \mathcal{F})  & C^1(\mathcal{M'}, \mathcal{F}) & 0 \\ 
        & \vdots & \vdots & \vdots &  \\ };
  
  { [start chain] \chainin (m-2-1);
\chainin (m-2-2);
\chainin (m-2-3);
\chainin (m-2-4);
\chainin (m-2-5); }
  
  { [start chain] \chainin (m-3-1);
\chainin (m-3-2);
\chainin (m-3-3);
\chainin (m-3-4);
\chainin (m-3-5); }

{ [start chain] \chainin (m-1-2);
\chainin (m-2-2);
\chainin (m-3-2);
\chainin (m-4-2);
 }

{ [start chain] \chainin (m-1-3);
\chainin (m-2-3);
\chainin (m-3-3);
\chainin (m-4-3);
 }

{ [start chain] \chainin (m-1-4);
\chainin (m-2-4);
\chainin (m-3-4);
\chainin (m-4-4);
 }
\end{tikzpicture}
\end{center}
which yields the long exact sequence of cohomology groups \cite[Thm.\ B.2.1]{Morimoto}
\begin{equation}
\begin{tikzpicture}
  \matrix (m) [matrix of math nodes, row sep=1em, column sep=2em]
    {0 & H^0(\mathcal{M}, \mathcal{M'}, \mathcal{F}) & H^0(\mathcal{M}, \mathcal{F})  & H^0(\mathcal{M'}, \mathcal{F}) \\
      \phantom 0 & H^1(\mathcal{M}, \mathcal{M'}, \mathcal{F}) & H^1(\mathcal{M}, \mathcal{F})  & H^1(\mathcal{M'}, \mathcal{F})  \\ 
      \phantom 0 & H^2(\mathcal{M}, \mathcal{M'}, \mathcal{F}) & \cdots. \phantom{000000} &  \\};
  
{ [start chain] \chainin (m-1-1);
\chainin (m-1-2);
\chainin (m-1-3);
\chainin (m-1-4);
 }
 { [start chain] \chainin (m-2-1);
\chainin (m-2-2);
\chainin (m-2-3);
\chainin (m-2-4);
 }
 { [start chain] \chainin (m-3-1);
\chainin (m-3-2);
\chainin (m-3-3); 
 }
\end{tikzpicture}
\label{longexact} 
\end{equation}

We can now formulate the main theorem of this subsection. We write $\mathcal{E}^*$ for the sheaf of ultradifferentiable functions of class $\ast$ on $\R^d$. 
\begin{theorem}\label{Cousin-scalar} Let $\Omega \subseteq \R^d$ be open and let $\mathcal{M} = \{ \Omega_i \, : \, i \in I\}$ be an open covering of $\Omega$. Then, $H^1(\mathcal{M}, \mathcal{E}^\ast) = 0$. Explicitly, this means that the sequence 
\begin{equation}
\begin{tikzpicture}
  \matrix (m) [matrix of math nodes, row sep=2em, column sep=2em]
    {0 & \mathcal{E}^{\ast}(\Omega)& \prod_{i \in I}\mathcal{E}^{\ast}(\Omega_i)  &Z^1(\mathcal{M}, \mathcal{E}^\ast) & 0 \\
   };
  
  { [start chain] \chainin (m-1-1);
\chainin (m-1-2);
\chainin (m-1-3) [join={node[above,labeled] {}}];
\chainin (m-1-4)[join={node[above,labeled] {\delta}}];
\chainin (m-1-5)[join={node[above,labeled] {}}];
}
  
\end{tikzpicture}
\label{Cousin}
\end{equation}
is exact, where
$$
Z^1(\mathcal{M}, \mathcal{E}^\ast) = \{ (\varphi_{i,j}) \in  \prod_{i,j \in I}\mathcal{E}^{\ast}(\Omega_{i,j}) \, : \, \varphi_{i,j} + \varphi_{j,k} + \varphi_{k,i} = 0 \mbox{ on } \Omega_{i,j,k},\quad \forall i,j,k \in I\},  
$$
and
$$
\delta = \delta_0: \prod_{i \in I}\mathcal{E}^{\ast}(\Omega_i) \to Z^1(\mathcal{M}, \mathcal{E}^\ast): (\varphi_i) \rightarrow ((\varphi_j - \varphi_{i})_{|\Omega_{i,j}}).
$$
\end{theorem}
We shall prove this theorem in several steps.
\begin{lemma}\label{bettersequence2}
Let $M_p$ be a weight sequence satisfying $(M.1)$, $(M.2)'$, and $p! \prec M_p$. Then, for every $r_j \in \mathcal{R}$ there is $r'_j \in \mathcal{R}$ with $r'_j \leq r_j$, $j \in \N$, such that the sequence $M_p/\prod_{j = 0}^pr'_j$ also satisfies  $(M.1)$, $(M.2)'$, and $p! \prec M_p$.
\end{lemma}
\begin{proof} Using \cite[Lemma 3.4]{Komatsu3},  we first find $k_j \in \mathcal{R}$ such that $p! \subset M_p/ \prod_{j = 0}^p k_j$. The sequence $r'_j \in \mathcal{R}$ with $r'_0 =  r'_1 = 1$ and 
$$
r'_j = \min\left( r_j,\frac{m_j}{m_{j-1}}r'_{j-1}, \sqrt{k_j}\right), \qquad j \geq 2,
$$
satisfies all requirements.
\end{proof}
\begin{proposition}\label{splitting-local}
Let $K_1$ and $K_2$ be compact sets in $\R^d$. The sequence
\begin{center}
\begin{tikzpicture}
  \matrix (m) [matrix of math nodes, row sep=2em, column sep=2em]
    {0 & \mathcal{E}^{\ast}[K_1 \cup K_2]& \mathcal{E}^{\ast}[K_1] \times \mathcal{E}^{\ast}[K_2]  & \mathcal{E}^{\ast}[K_1 \cap K_2] & 0 \\
   };
  
  { [start chain] \chainin (m-1-1);
\chainin (m-1-2);
\chainin (m-1-3) [join={node[above,labeled] {}}];
\chainin (m-1-4)[join={node[above,labeled] {}}];
\chainin (m-1-5)[join={node[above,labeled] {}}];
}
  
\end{tikzpicture}
\end{center}
is exact.
\end{proposition}
\begin{proof} We only need to show the surjectivity of the mapping
$$
\mathcal{E}^{\ast}[K_1] \times \mathcal{E}^{\ast}[K_2]  \rightarrow \mathcal{E}^{\ast}[K_1 \cap K_2]: (\varphi_1, \varphi_2) \rightarrow \varphi_2 -\varphi_1,
$$
the rest is clear.

\emph{Roumieu case}: The transpose of the above mapping is given by
$$
\mathcal{E}'^{\{M_p\}}[K_1 \cap K_2] \rightarrow \mathcal{E}'^{\{M_p\}}[K_1] \times \mathcal{E}'^{\{M_p\}}[K_2]: f \rightarrow (-f, f).
$$
The result is therefore a consequence of the fact that this mapping is injective and has closed range, as follows from Proposition \ref{support}.

\emph{Beurling case}: 
Let $\varphi \in \mathcal{E}^{(M_p)}[K_1 \cap K_2]$. By \cite[Lemma 3.4]{Komatsu3},  there is $r_j \in \mathcal{R}$ such that $\varphi \in \mathcal{E}^{\{M_p/ \prod_{j = 0}^p r_j\}}[K_1 \cap K_2]$. By Lemma \ref{bettersequence2}, we may assume without loss of generality that $M_p/\prod_{j = 0}^pr_j$ satisfies $(M.1)$, $(M.2)'$, and $p! \prec M_p$. The result now follows from the Roumieu case.

\end{proof}

\begin{proof}[Proof of Theorem \ref{Cousin-scalar}]
STEP $1$: $I= \{ 1,2\}$. Let $(\Omega_{j,n})_{n \in \N}$ be an exhaustion by relatively compact open subsets of $\Omega_j$, $j = 1,2$. Define the following projective spectra
$$
\mathcal{X} = ( \mathcal{E}^{\ast}[\overline{\Omega}_{1,n} \cup \overline{\Omega}_{2,n}])_{n \in \N}, \qquad \mathcal{Y} = ( \mathcal{E}^{\ast}[\overline{\Omega}_{1,n}] \times \mathcal{E}^\ast[\overline{\Omega}_{2,n}])_{n \in \N}, \qquad \mathcal{Z} = ( \mathcal{E}^{\ast}[\overline{\Omega}_{1,n} \cap \overline{\Omega}_{2,n}])_{n \in \N}.
$$
By Proposition \ref{splitting-local}, we have the following  exact sequence of projective spectra
\begin{center}
\begin{tikzpicture}
  \matrix (m) [matrix of math nodes, row sep=2em, column sep=2em]
    {0 & \mathcal{X} & \mathcal{Y} & \mathcal{Z} & 0. \\
   };
  { [start chain] \chainin (m-1-1);
\chainin (m-1-2);
\chainin (m-1-3) [join={node[above,labeled] {}}];
\chainin (m-1-4)[join={node[above,labeled] {}}];
\chainin (m-1-5)[join={node[above,labeled] {}}];
}

\end{tikzpicture}
\end{center}
Since
$$
\operatorname{Proj}^0 \mathcal{X} \cong \mathcal{E}^\ast(\Omega_1 \cup \Omega_2), \qquad \operatorname{Proj}^0 \mathcal{Y} \cong \mathcal{E}^\ast(\Omega_1) \times \mathcal{E}^\ast(\Omega_2), \qquad \operatorname{Proj}^0 \mathcal{Z} \cong \mathcal{E}^\ast(\Omega_1 \cap \Omega_2),
$$
it suffices to show that $\operatorname{Proj}^1 \mathcal{X} =0$.

\emph{Beurling case}: The spectrum $\mathcal{X}$ is equivalent (in the sense of \cite[Def.\ 3.1.6]{Wengenroth2}) to the spectrum  $\mathcal{X}_0 =(\mathcal{E}^{(M_p)}(\Omega_{1,n} \cup \Omega_{2,n}))_{n \in \N}$. Hence, by \cite[Prop.\ 3.1.7]{Wengenroth2}, we have $\operatorname{Proj}^1 \mathcal{X} \cong \operatorname{Proj}^1 \mathcal{X}_0$. Since the spectrum $\mathcal{X}_0$ consists of Fr\'echet spaces and is reduced, the Mittag-Leffler lemma (see e.g.\ \cite[Lemma 1.3]{Komatsu}, \cite[Thm.\ 3.2.1]{Wengenroth2}) implies that $\ \operatorname{Proj}^1 \mathcal{X}_0 = 0$.

\emph{Roumieu case}: Immediate consequence of Propositions \ref{VogtandWeng} and \ref{topologytestspaces}.

STEP $2$: $I$ is finite. This can be shown by using the first step and an induction argument (for details see the second step in the the proof of \cite[Thm.\ 2.3.1]{Morimoto}).

STEP $3$: $I$ is arbitrary. Since every open set of $\R^d$ is second countable, we may assume without loss of generality that $I$ is countable (set $I = \N$) and that $\Omega_i \Subset \Omega$ for all $i \in \N$.
 Define the following projective spectra
 \begin{equation}
\mathcal{X} = \left( \mathcal{E}^{\ast}\left(\bigcup_{i = 0}^n\Omega_i\right)\right)_{n \in \N}, \quad \mathcal{Y} = \left(\prod_{i = 0}^n \mathcal{E}^{\ast}(\Omega_i)\right)_{n \in \N}, \quad \mathcal{Z} = ( Z^1(\mathcal{M}_n,\mathcal{E}^{\ast}))_{n \in \N},
\label{spectra-top-exactness}
\end{equation}
where $\mathcal{M}_n = \{ \Omega_i \, : \, i = 0, \ldots, n \}$.
By the second step, we have the following exact sequence of projective spectra
\begin{equation}
\begin{tikzpicture}
  \matrix (m) [matrix of math nodes, row sep=2em, column sep=2em]
    {0 & \mathcal{X} & \mathcal{Y} & \mathcal{Z} & 0 \: . \\
   };
  { [start chain] \chainin (m-1-1);
\chainin (m-1-2);
\chainin (m-1-3) [join={node[above,labeled] {}}];
\chainin (m-1-4)[join={node[above,labeled] {}}];
\chainin (m-1-5)[join={node[above,labeled] {}}];
}
\end{tikzpicture}
\label{top-exact-seq}
\end{equation}
Since
$$
\operatorname{Proj}^0 \mathcal{X} \cong \mathcal{E}^\ast(\Omega), \qquad \operatorname{Proj}^0 \mathcal{Y} \cong \prod_{i \in \N}\mathcal{E}^\ast(\Omega_i), \qquad \operatorname{Proj}^0 \mathcal{Z} \cong Z^1(\mathcal{M},\mathcal{E}^{\ast}),
$$
it is enough to verify that $\operatorname{Proj}^1 \mathcal{X} =0$.

\emph{Beurling case}: Since the spectrum $\mathcal{X}$ consists of Fr\'echet spaces and is reduced, it follows again from the Mittag-Leffler lemma.

\emph{Roumieu case}: The spectrum $\mathcal{X}$ is equivalent to the spectrum $\mathcal{X}_0 = (\mathcal{E}^{\{M_p\}}[\bigcup_{i = 0}^n\overline{\Omega}_i])_{n \in \N}$. By Propositions \ref{VogtandWeng} and \ref{topologytestspaces}, we have $\operatorname{Proj}^1 \mathcal{X} \cong \operatorname{Proj}^1 \mathcal{X}_0 = 0$.
\end{proof}
\begin{remark}
\label{rmkHorCousin} It is worth comparing Theorem \ref{Cousin-scalar} with H\"{o}rmander's work \cite{Hormander} in the Roumieu case. Since Proposition \ref{splitting-local}  in this case follows directly from H\"{o}rmander's support theorem (Proposition \ref{support}), one may say that it is implicitly contained in his work. By merely combinatorial means, one easily deduces from Proposition \ref{splitting-local} that given finitely many compact sets $K_1,K_2,\dots, K_n$ in $\mathbb{R}^{d}$ and germs of quasianalytic functions $\varphi_{i,j}\in \mathcal{E}^{\{M_p\}}[K_{i}\cap K_j]$, subject to the co-cycle conditions $\varphi_{i,j} + \varphi_{j,k} + \varphi_{k,i} = 0$, there are germs $\varphi_{i}\in \mathcal{E}^{\{M_p\}}[K_{i}]$ such that $\varphi_{i,j}=\varphi_{j}-\varphi_{i}$. In the passage to open sets and (finite or infinite) open coverings, the essential ingredients for Theorem \ref{Cousin-scalar} are then Propositions \ref{VogtandWeng} and \ref{topologytestspaces}.  
\end{remark}

Next, we discuss the topological exactness of the sequence \eqref{Cousin}. We endow $\prod \mathcal{E}^{\ast}(\Omega_i)$ with the product topology and  $Z^1(\mathcal{M}, \mathcal{E}^\ast)$ with the relative topology induced by
$\prod \mathcal{E}^{\ast}(\Omega_{i,j})$ (endowed with the product topology). Notice that  $Z^1(\mathcal{M}, \mathcal{E}^\ast)$ is a closed subspace of $\prod\mathcal{E}^{\ast}(\Omega_{i,j})$ and that the mapping $\delta$ is continuous.
\begin{proposition}\label{Cousin-topology} The sequence \eqref{Cousin} is topologically exact if $I$ is countable.
\end{proposition}
In the Roumieu case, we need the ensuing lemmas. 
\begin{lemma} \label{lemma1}
Let $X$ be a topological space and let $\mathcal{F}$ be a sheaf on $X$. Suppose that $H^1(\mathcal{M},\mathcal{F}) = 0$ for all finite open coverings $\mathcal{M}$. Then, $H^p(\mathcal{M},\mathcal{F}) = 0$ for all $p \geq 1$ and all finite open coverings $\mathcal{M}$.
\end{lemma}
\begin{proof} We use induction on  $N = |\mathcal{M}|$. The case $N = 1$ is clear. Suppose that the result holds for $N$. Let $\mathcal{M} = \{\Omega_i \, : \, i = 0, \ldots, N\}$ be an arbitrary open covering and define $\mathcal{M}' = \{\Omega_i \, : \, i = 0, \ldots, N-1\}$. Using the induction hypothesis, we obtain $H^p(\mathcal{M}',\mathcal{F}) = 0$ for all $p \geq 1$. Hence, the long exact sequence of cohomology groups \eqref{longexact} implies that $H^p(\mathcal{M},\mathcal{F}) \cong H^p(\mathcal{M}, \mathcal{M}', \mathcal{F})$ for all $p \geq 2$. A straightforward calculation yields $C^p(\mathcal{M}, \mathcal{M}',\mathcal{F}) \cong C^{p-1}(\widetilde{\mathcal{M}},\mathcal{F})$ for all $p \geq 1$, where  $\widetilde{\mathcal{M}} = \{\Omega_i \cap \Omega_n \, : \, i = 0, \ldots, N-1\}$. Therefore, we have $H^p(\mathcal{M},\mathcal{F}) \cong H^p(\mathcal{M}, \mathcal{M}', \mathcal{F}) \cong H^{p-1}(\widetilde{\mathcal{M}},\mathcal{F}) =  0$, for all $p \geq 2$, where in the last inequality we have used the induction hypothesis.
\end{proof}

\begin{lemma} \label{lemma2}
Let 
\begin{center}
\begin{tikzpicture}
  \matrix (m) [matrix of math nodes, row sep=2em, column sep=2em]
    {0 & X_0 & X_1   &\cdots & X_N & 0 \\
   };
  
  { [start chain] \chainin (m-1-1);
\chainin (m-1-2);
\chainin (m-1-3) [join={node[above,labeled] {d_0}}];
\chainin (m-1-4)[join={node[above,labeled] {d_1}}];
\chainin (m-1-5)[join={node[above,labeled] {d_{N-1}}}];
\chainin (m-1-6)[join={node[above,labeled] {}}];
}

\end{tikzpicture}
\end{center}
be an exact sequence of  ultrabornological $(PLS)$-spaces. Then, the sequence is automatically topologically exact and $\ker d_j$ is an ultrabornological $(PLS)$-space for each $j = 0, \ldots, N-1$.
\end{lemma}
\begin{proof}
Since every $(PLS)$-space $X$  has a strict ordered web, De Wilde's open mapping theorem \cite[Thm.\ 24.30]{Meise} implies that any linear continuous surjective mapping $X \rightarrow Y$, with $Y$ ultrabornological, is a topological homomorphism. Moreover, a closed subspace $A$ of an ultrabornological $(PLS)$-space $X$ is ultrabornological if and only if $X/ A$ is complete \cite[Cor.\ 1.4]{Domanski}. Combining these two facts, we obtain the desired result.
\end{proof}
\begin{proof}[Proof of Proposition \ref{Cousin-topology}] In the Beurling case, the statement is a consequence of the open mapping theorem for Fr\'{e}chet spaces. We now consider the Roumieu case. Since the countable product of $(PLS)$-spaces is a $(PLS)$-space and a closed subspace of a $(PLS)$-space is again a $(PLS)$-space, the spaces appearing in \eqref{Cousin} are all $(PLS)$-spaces. We divide the proof into two steps.

STEP $1$: $I$ is finite. Suppose $I= \{ 0, \ldots, N\}$ for some $N \in \N$. Theorem \ref{Cousin-scalar} and Lemma \ref{lemma1} imply that the sequence 
\begin{center}
\begin{tikzpicture}
  \matrix (m) [matrix of math nodes, row sep=2em, column sep=2em]
    {0 & \mathcal{E}^{\ast}(\Omega)& \prod_{i =0}^N\mathcal{E}^{\ast}(\Omega_i)  &C^1(\mathcal{M}, \mathcal{E}^\ast) & \cdots & C^N(\mathcal{M}, \mathcal{E}^\ast) &  0 \\
   };
  
  { [start chain] \chainin (m-1-1);
\chainin (m-1-2);
\chainin (m-1-3) [join={node[above,labeled] {}}];
\chainin (m-1-4)[join={node[above,labeled] {}}];
\chainin (m-1-5)[join={node[above,labeled] {}}];
\chainin (m-1-6)[join={node[above,labeled] {}}];
\chainin (m-1-7)[join={node[above,labeled] {}}];

}  
\end{tikzpicture}
\end{center}
is exact. Notice that $C^p(\mathcal{M}, \mathcal{E}^\ast)$ is isomorphic to a finite product of spaces of the form $\mathcal{E}^{\ast}(\Omega_{i_0, \ldots, i_p})$, $0 \leq i_0 < i_1 < \cdots < i_p \leq N$ . We endow it with the product topology. In such a way the linking mappings in the above sequence become continuous. Moreover, since a finite product of ultrabornological spaces is again an ultrabornological space, the result follows from Proposition \ref{topologytestspaces} and Lemma \ref{lemma2}.

STEP $2$: $I$ is countable (set $I = \N$). Consider the projective spectra defined in \eqref{spectra-top-exactness}. By the first step we know that the the complex \eqref{top-exact-seq} consists of topologically exact sequences. Since every $(PLS)$-space $X$  has a strict ordered web and $\operatorname{Proj}^1 \mathcal{X} =0$ (see the third step in the proof of Theorem \ref{Cousin-scalar}), \cite[Thm.\  3.3]{Wengenroth2} implies that the mapping $\delta$ appearing in \eqref{Cousin} is a topological homomorphism. As the countable product of ultrabornological spaces is again an ultrabornological space and the quotient of an ultrabornological space with a closed subspace is again ultrabornological, we obtain that $Z^1(\mathcal{M}, \mathcal{E}^{\{M_p\}})$ is an ultrabornological $(PLS)$-space from Proposition \ref{topologytestspaces}. Furthermore, it implies that $\operatorname{ker}\delta$ is ultrabornological (cf. the proof of Lemma \ref{lemma2}). Hence  $\mathcal{E}^{\ast}(\Omega) \rightarrow \prod_{i \in \N}\mathcal{E}^{\ast}(\Omega_i)$ is a topological embedding by De Wilde's open mapping theorem.

\end{proof}
\subsection{Vector-valued case}
We now address the  Cousin problem for spaces of $F$-valued quasianalytic functions for suitable l.c.s.\ $F$. We write $\mathcal{E}^*(\cdot \, ; F)$ for the sheaf of $F$-valued ultradifferentiable functions of class $\ast$ on $\R^d$. 
\begin{theorem}\label{Cousin-vector} Let $\Omega \subseteq \R^d$ be open, let $\mathcal{M} = \{ \Omega_i \, : \, i \in I\}$ be an open covering of $\Omega$, and let $F$ be a locally convex space. Then, $H^1(\mathcal{M}, \mathcal{E}^*(\cdot \, ; F)) = 0$ in the following cases:
\begin{itemize}
\item[$(i)$] For $\ast = (M_p)$ and $F$ a Fr\'echet space,
\item[$(ii)$] for $\ast = \{M_p\}$ and  $F$ a $(DFS)$-space such that $F'$ satisfies $(\underline{DN})$,
\item[$(iii)$] for $\ast = \{p!\}$ and $F$ a $(PLS)$-space satisfying the dual interpolation estimate for small theta.
\end{itemize}
\end{theorem}

We shall need the following lemma in the proof of the Roumieu case.
\begin{lemma}\cite[Prop.\ 4.5]{Domanski2} \label{exacttensor}
Let
\begin{center}
\begin{tikzpicture}
  \matrix (m) [matrix of math nodes, row sep=2em, column sep=2em]
    {0 & X & Y & Z & 0 \\
   };
  { [start chain] \chainin (m-1-1);
\chainin (m-1-2);
\chainin (m-1-3) [join={node[above,labeled] {S}}];
\chainin (m-1-4)[join={node[above,labeled] {T}}];
\chainin (m-1-5)[join={node[above,labeled] {}}];
}

\end{tikzpicture}
\end{center}
be a topologically exact sequence of $(PLS)$-spaces and let $F$ be a $(PLS)$-space. Suppose that $X$ is a $(PLN)$-space and that the $(PLS)$-space $X\varepsilon F$ is ultrabornological. Then, the sequence
\begin{center}
\begin{tikzpicture}
  \matrix (m) [matrix of math nodes, row sep=2em, column sep=2em]
    {0 & X\varepsilon F & Y\varepsilon F & Z\varepsilon F & 0 \\
   };
  { [start chain] \chainin (m-1-1);
\chainin (m-1-2);
\chainin (m-1-3) [join={node[above,labeled] {S\varepsilon \operatorname{id}_F}}];
\chainin (m-1-4)[join={node[above,labeled] {T\varepsilon \operatorname{id}_F}}];
\chainin (m-1-5)[join={node[above,labeled] {}}];
}

\end{tikzpicture}
\end{center}
is exact.
\end{lemma}
\begin{proof}[Proof of Theorem \ref{Cousin-vector}] As in the scalar-valued case, the vanishing of the first cohomology group $H^1(\mathcal{M}, \mathcal{E}^*(\cdot \, ; F))$ means that
\begin{center}
\begin{tikzpicture}
  \matrix (m) [matrix of math nodes, row sep=2em, column sep=2em]
    {0 & \mathcal{E}^{\ast}(\Omega;F)& \prod_{i \in I}\mathcal{E}^{\ast}(\Omega_i;F)  &Z^1(\mathcal{M}, \mathcal{E}^\ast(\cdot,F)) & 0 \\
   };
 { [start chain] \chainin (m-1-1);
\chainin (m-1-2);
\chainin (m-1-3) [join={node[above,labeled] {}}];
\chainin (m-1-4)[join={node[above,labeled] {\delta_F}}];
\chainin (m-1-5)[join={node[above,labeled] {}}];
}
  
\end{tikzpicture}
\end{center}
is exact, where
$$
Z^1(\mathcal{M}, \mathcal{E}^\ast(\cdot,F)) = \{ (\bm{\varphi}_{i,j}) \in  \prod_{i,j \in I}\mathcal{E}^{\ast}(\Omega_{i,j};F)  : \bm{\varphi}_{i,j} + \bm{\varphi}_{j,k} + \bm{\varphi}_{k,i} = 0 \mbox{ on } \Omega_{i,j,k},  \forall i,j,k \in I\},  
$$
and
$$
\delta_F = \prod_{i \in I}\mathcal{E}^{\ast}(\Omega_i;F) \to Z^1(\mathcal{M}, \mathcal{E}^\ast(\cdot,F)): (\bm{\varphi}_i) \rightarrow ((\bm{\varphi}_j- \bm{\varphi}_{i})_{|\Omega_{i,j}}).
$$
We may assume that $I$ is countable. Furthermore, we only need to show that $\delta_F$ is surjective, the rest is clear. Notice that we have the following isomorphisms of l.c.s.
\begin{align}\label{reprcousin}
&\mathcal{E}^{\ast}(\Omega;F) \cong \mathcal{E}^{\ast}(\Omega) \varepsilon F, \qquad \prod_{i \in I}\mathcal{E}^{\ast}(\Omega_i;F) \cong \left( \prod_{i \in I}\mathcal{E}^{\ast}(\Omega_i)\right) \varepsilon F, \nonumber \\
&Z^1(\mathcal{M}, \mathcal{E}^\ast(\cdot,F)) \cong Z^1(\mathcal{M}, \mathcal{E}^\ast) \varepsilon F.
\end{align}
The first two isomorphisms are consequences of Proposition \ref{charvectorvaluedfunctions} and \cite[Prop.\ 1.5]{Komatsu3}. We now show the third one. Since the $\varepsilon$-product of two injective linear topological homomorphisms is again an injective linear topological homomorphism, the space $Z^1(\mathcal{M}, \mathcal{E}^\ast) \varepsilon F$ is topologically isomorphic to a subspace $X$ of 
$$
\left(\prod_{i,j \in I}\mathcal{E}^{\ast}(\Omega_{i,j})\right)\varepsilon F  \cong \prod_{i,j \in I}\mathcal{E}^{\ast}(\Omega_{i,j};F). 
$$
Let us now prove that $X = Z^1(\mathcal{M}, \mathcal{E}^\ast(\cdot\, ;F))$. Let $(\bm{\varphi}_{i,j}) \in X$. Employing the representation $Z^1(\mathcal{M}, \mathcal{E}^\ast) \varepsilon F \cong L(F'_c, Z^1(\mathcal{M}, \mathcal{E}^\ast))$  we obtain that 
$$
(\langle y', \bm{\varphi}_{i,j}(\cdot) \rangle) \in Z^1(\mathcal{M}, \mathcal{E}^\ast), \qquad y' \in F',
$$
and, thus,
$$
0 = \langle y', \bm{\varphi}_{i,j}(x) \rangle + \langle y', \bm{\varphi}_{j,k}(x) \rangle + \langle y', \bm{\varphi}_{k,i}(x) \rangle = \langle y', \bm{\varphi}_{i,j}(x) +\bm{\varphi}_{j,k}(x) + \bm{\varphi}_{k,i}(x)\rangle, $$
for all $ y' \in F'$, $x \in \Omega_{i,j,k}$,
$i,j,k \in I$. This implies that $(\bm{\varphi}_{i,j}) \in Z^1(\mathcal{M}, \mathcal{E}^\ast(\cdot\, ;F))$. The converse inclusion can be shown similarly.  Finally, notice that $\delta_F = \delta \varepsilon \operatorname{id}_F$.

$(i)$: By the hereditary properties of nuclearity, we have that the spaces $\prod_{i \in I}\mathcal{E}^{(M_p)}(\Omega_i)$ and  $Z^1(\mathcal{M}, \mathcal{E}^{(M_p)})$ are nuclear Fr\'echet spaces. Hence, by \eqref{reprcousin}, we may represent $\delta_F$ as a tensor product of mappings,
$$
\delta_F = \delta \widehat{\otimes}\operatorname{id}_F :  \left( \prod_{i \in I}\mathcal{E}^{(M_p)}(\Omega_i)\right)  \widehat{\otimes} F \rightarrow Z^1(\mathcal{M}, \mathcal{E}^{(M_p)}) \widehat{\otimes} F.
$$
The result now follows from the solution to the scalar-valued Cousin problem (Theorem \ref{Cousin-scalar}) and the following well known fact: Given two surjective continuous linear mappings $T_1: X_1 \rightarrow Y_1$ and $T_2: X_2 \rightarrow Y_2$ between Fr\'echet spaces, the mapping 
$$
T_1 \widehat{\otimes}_\pi T_2: X_1 \widehat{\otimes}_\pi X_2 \rightarrow Y_1 \widehat{\otimes}_\pi Y_2
$$
is also surjective.

$(ii)$ and $(iii)$: In view of Lemma \ref{exacttensor} this follows directly from Proposition \ref{topologyvectorvalued} (Remark \ref{realanalytic-top} in the real analytic case) and Proposition \ref{Cousin-topology}.
\end{proof}


\begin{thebibliography}{99}
\setlength{\itemsep}{0pt}

\bibitem{Bierstedt} K.~D.~Bierstedt, R.~Meise, \emph{Induktive Limites gewichteter R\"aume stetiger und holomorpher Funktionen}, J. Reine Angew. Math. \textbf{282} (1976), 186--220.

\bibitem{Bonet00}  J.~Bonet, P.~Doma\'nski, \emph{Real analytic curves in Fr\'echet spaces and their duals},  Monatsh. Math. \textbf{126} (1998),  13--36.

\bibitem{Bonet}  J.~Bonet, P.~Doma\'nski, \emph{The structure of spaces of quasianalytic functions of Roumieu type}, Arch. Math. (Basel) \textbf{89} (2007), 430--441.

\bibitem{BDV}  J.~Bonet, P.~Doma\'nski, D.~Vogt, \emph{Interpolation of vector valued real analytic functions}, J. London Math. Soc. \textbf{66} (2002), 407--420.
 
 \bibitem{Bonet2} J.~Bonet, P.~P\'erez Carreras, \emph{Barrelled locally convex spaces}, North-Holland Publishing Co., Amsterdam, 1987.
 
\bibitem{Braun} R.~W.~Braun, R.~Meise, B.~A.~Taylor, \emph{Ultradifferentiable functions and Fourier analysis}, Result. Math. \textbf{17} (1990), 206--237.

\bibitem{Chung}S.~Y.~Chung, D.~Kim, \emph{Between hyperfunctions and ultradistributions}, Bull. Cl. Sci. Math. Nat. Sci. Math. \textbf{24} (1999), 89--106.

\bibitem{DVV} A.~Debrouwere, H.~Vernaeve, J.~Vindas, \emph{Optimal embeddings of ultradistributions into differential algebras}, preprint (arXiv:1601.03512).

\bibitem{DVVinfra} A.~Debrouwere, H.~Vernaeve, J.~Vindas, \emph{A non-linear theory of infrahyperfunctions,} preprint (arXiv:1701.06996).

\bibitem{Domanski} P.~Doma\'nski, D.~Vogt, \emph{A splitting theory for the space of distributions}, Studia Math. \textbf{140} (2000), 57--77. 

\bibitem{Domanski2} P.~Doma\'nski, \emph{Real analytic parameter dependence of solutions of differential equations}, Rev. Mat. Iberoam. \textbf{26} (2010),  175--238.

\bibitem{Domanski4} P.~Doma\'nski, L.~Frerick, D.~Vogt, \emph{Fr\'echet quotients of spaces of real analytic functions}, Studia Math. \textbf{159} (2003), 229--245.  

\bibitem{Domanski3} P.~  Doma\'nski, D.~Vogt, \emph{The space of real analytic functions has no basis}, Studia Math. \textbf{142} (2000), 187--200.

\bibitem{Gunning} R.~C.~Gunning, H.~Rossi, \emph{Analytic functions of several complex variables}, AMS Chelsea Publishing, Providence, RI, 2009.
 
\bibitem{Heinrich} T.~Heinrich, R.~Meise, \emph{A support theorem for quasianalytic functionals}, Math. Nachr. \textbf{280} (2007), 364--387.

\bibitem{Hormander} L.~H\"ormander,  \emph{Between distributions and hyperfunctions}, Ast\'{e}risque \textbf{131} (1985), 89--106. 

\bibitem{Hormander2} L.~H\"ormander, \emph{The analysis of linear partial differential operators. I. Distribution
theory and Fourier analysis}, Springer-Verlag, Berlin, 1990.

\bibitem{Komatsu} H.~Komatsu, \emph{Ultradistributions I. Structure theorems and a characterization}, J. Fac. Sci. Univ. Tokyo Sect. IA Math. \textbf{20} (1973), 25--105.  

\bibitem{Komatsu3} H.~Komatsu, \emph{Ultradistributions. III. Vector-valued ultradistributions and the theory of kernels}, J. Fac. Sci. Univ. Tokyo Sect. IA Math. \textbf{29} (1982), 653--717.

\bibitem{Kothe} G.~K\"othe, \emph{Topological vector spaces. I}, Springer-Verlag New York Inc., New York, 1969.

\bibitem{K-M} A.~Kriegl, P.~ W.~Michor, \emph{The convenient setting of global analysis}, American Mathematical Society, Providence, RI, 1997.

\bibitem{K-M-R} A.~Kriegl, P.~ W.~Michor, A.~Rainer, \emph{The convenient setting for quasianalytic Denjoy-Carleman differentiable mappings,} J. Funct. Anal. \textbf{261} (2011), 1799--1834.


\bibitem{Martineau2} A.~Martineau, \emph{Sur la topologie des espaces de fonctions holomorphes}, Math. Ann. \textbf{163} (1966), 62--88.

\bibitem{Matsuzawa} T.~Matsuzawa, \emph{ A calculus approach to hyperfunctions I}, Nagoya Math. J. \textbf{108} (1987), 53--66.

\bibitem{Meise} R.~Meise, D.~Vogt, \emph{Introduction to functional analysis}, The Clarendon Press, Oxford University Press, New York, 1997.

\bibitem{Morimoto} M.~Morimoto, \emph{An introduction to Sato's hyperfunctions}, AMS, Providence, RI, 1993.

\bibitem{Palamodov} V.~P.~Palamodov,  \emph{The projective limit functor in the category of linear topological spaces}, Math. USSR-Sb. \textbf{4} (1968), 529--559.

\bibitem{Petzsche1984} H.~J.~Petzsche, \emph{Generalized functions and the boundary values of holomorphic functions,} J. Fac. Sci. Univ. Tokyo Sect. IA Math. \textbf{31} (1984), 391--431.

\bibitem{Piszczek} K.~Piszczek, \emph{On a property of PLS-spaces inherited by their tensor products}, Bull. Belg. Math. Soc. Simon Stevin \textbf{17} (2010),
155--170.

\bibitem{Rosner}  T.~R\"osner, \emph{Surjektivit\"at partieller Differentialoperatoren auf quasianalytischen Roumieu-Klassen}, Dissertation, D\"usseldorf, 1997.


\bibitem{Schwartz} L.~Schwartz, \emph{Th\'eorie des distributions \`{a} valeurs vectorielles. I}, Ann. Inst. Fourier (Grenoble) \textbf{7} (1957), 1--141. 

\bibitem{Takiguchi} T.~Takiguchi, \emph{Structure of quasi-analytic ultradistributions}, Publ. Res. Inst. Math. Sci. \textbf{43} (2007), 425--442.

\bibitem{Treves} F.~Tr\`{e}ves, \emph{Topological vector spaces, distributions and kernels}, Academic Press, New York, 1967.

\bibitem{Wengenroth} J.~Wengenroth, \emph{Acyclic inductive spectra of Fr\'echet spaces}, Studia Math. \textbf{120} (1996), 247--258.

\bibitem{Wengenroth2} J.~Wengenroth, \emph{Derived functors in functional analysis}, Springer-Verlag, Berlin, 2003.

\end{thebibliography}
\end{document}